\documentclass[letterpaper, 10 pt, conference]{IEEEtran} 

\IEEEoverridecommandlockouts 

\usepackage{cite}
\usepackage[utf8]{inputenc} 
\usepackage[T1]{fontenc}    
\usepackage{hyperref}       
\usepackage{url}            
\usepackage{booktabs}       
\usepackage{amsfonts}       
\usepackage{nicefrac}       
\usepackage{microtype}      
\usepackage{xcolor}         
\usepackage{subcaption}
\usepackage{tabularx}
\usepackage{multirow}

\usepackage{xcolor}
\usepackage[pdftex]{graphicx}
\usepackage{amsmath}
\usepackage{algorithm}
\usepackage{algorithmic}

\usepackage{amsthm}
\renewenvironment{proof}[1][Proof]{\noindent\hspace{2em}{\itshape #1: }}{\hspace*{\fill}~\QED\par\endtrivlist\unskip}

\newtheorem{theorem}{$\textbf{Theorem}$}

\newtheorem{lemma}{$\textbf{Lemma}$}
\newtheorem{definition}{$\textbf{Definition}$}

\newcolumntype{C}{>{\centering\arraybackslash}X}
\newcolumntype{P}[1]{>{\raggedright\arraybackslash}p{#1}}

\title{\LARGE \bf Efficient Quantification of Time-Series Prediction Error: Optimal Selection Conformal Prediction}

\author{Boyu Pang and Kostas Margellos
\thanks{The authors are with the Department of Engineering Science, University of Oxford, Oxford, United Kingdom. E-mails: \{boyu.pang, kostas.margellos\}@eng.ox.ac.uk  } 
\thanks{For the purpose of Open Access, the authors have applied a CC BY public copyright licence to any Author Accepted Manuscript (AAM) version arising from this submission. }}

\begin{document}

\maketitle
\thispagestyle{empty}
\pagestyle{empty}

\begin{abstract}
Designing effective score functions in Conformal Prediction (CP) for time-series data remains challenging due to conservativeness and/or computational inefficiency.  We propose Optimal Selection Conformal Prediction (OSCP), which parameterizes the score function via offset terms. To determine these parameters, we formulate a mixed-integer linear program (MILP) that minimizes an empirical proxy of the region size. We further reformulate this optimization problem into a smaller form (fewer constraints) to improve computational efficiency. We provide theoretical guarantees on both validity and CP-efficiency of OSCP. Numerical experiments demonstrate that OSCP reduces uncertainty-set size  and has much lower computational requirements compared to the state-of-the-art method.
\end{abstract}

\section{Introduction} \label{sec: Introduction}
Uncertainty is almost ubiquitous in safety-critical autonomous systems. The dynamic nature of external environments (e.g., autonomous driving) and the incorporation of learning-based methods (e.g., neural-networks) introduce uncertainties into the systems, which pose new challenges to safe controller design and verification. To address this issue, one way is to quantify uncertainty, in particular, for the time-series predictions that arise in multi-stage optimization problems. We classify uncertainty quantification methods into two main streams:

\paragraph{Bayesian methods and concentration-bounds}  Uncertainty quantification using Bayesian methods include Bayesian Inference, Bayesian Neural Network and other variants \cite{fortunato2017bayesian, neal2011mcmc, kingma2013auto}. Alternative approaches include concentration bounds, such as Chernoff-Hoeffding (e.g., \cite{legay2019statistical}), Clopper-Pearson (e.g., \cite{wang2019statistical}). A more detailed review can be found in Section 2.3 of the survey paper \cite{lindemann2024formal}. However, Bayesian methods do not have finite-sample guarantees and become computationally intractable for large-scale problems; concentration bounds are usually conservative in this task. An alternative to these methods is conformal prediction, which is presented next.
\paragraph{Conformal Prediction} 
Conformal Prediction (CP) \cite{vovk2005algorithmic} has emerged as a lightweight, distribution-free alternative that provides $(1-\epsilon)$ confidence regions with rigorous finite-sample validity. Under mild assumptions on a calibration dataset, CP provides tight finite-sample guarantees. Although closely related with scenario optimization \cite{campi2018introduction}, a tool that has also been used for tight uncertainty quantification \cite{margellos2014road}, CP focuses on a different aspect and thus complements the scenario approach (see \cite{o2025bridging} for some connections between the two methods).  Thanks to CP's simplicity, flexibility, and computational-efficiency, CP has been applied in probabilistic safe control synthesis and verification, such as moving-objects avoidance control \cite{lindemann2023safe-a,stamouli2024recursively,tonkens2023scalable,yu2023signal}, probabilistic reachability analysis \cite{bortolussi2019neural,cairoli2021neural,cairoli2023conformal}, probabilistic reachable sets construction \cite{hashemi2023data,hashemi2024statistical,tebjou2023data}. 

Although  CP offers  rigorous probabilistic guarantees, its performance (e.g., size and shape of the region)  depends critically on one of its core components, the \emph{score function}  (also called nonconformity-measure). Designing an appropriate score function for time-series uncertainty quantification is non-trivial. Existing works \cite{stankeviciute2021conformal, sun2024copula, Tumu2024multi-modal, cleaveland2024conformal} either suffer from 1) overly conservative confidence regions \cite{stankeviciute2021conformal}, 2) overfitting or fitting-errors \cite{sun2024copula, Tumu2024multi-modal}, or 3) computational intractability issues \cite{Tumu2024multi-modal,cleaveland2024conformal}.
To the best of our knowledge, no work seems to alleviate these issues at the same time.

In this paper, we aim to overcome these challenges when using CP for time series, thus opening the road for its use in multi-stage optimization and safety problems that require uncertainty quantification over entire trajectories rather than single time-steps. We propose a new parameterized score function that can be optimized to provide empirical minimal-average-radius CP regions. Our CP method generates norm-ball regions that are convex and, as we will show, also tight for multi-dimensional time series, and it exhibits lower computational requirements compared to other algorithmic alternatives. Our proposed approach is directly applicable to control problems such as safe learning-based MPC \cite{lindemann2023safe-a,stamouli2024recursively,tonkens2023scalable,yu2023signal} and multi-stage safety verification \cite{Lindemann2023Conformal-b}.

Our main contributions can be summarized as:

\begin{enumerate}
    \item \textbf{Optimization-based Score Design:} We propose a new parameterized score function for calibrating multi-dimensional time-series data in CP, and a mixed-integer linear programming (MILP) problem to determine optimal parameter solutions. We then provide a re-formulation of this MILP with fewer constraints to enable faster computation.
    
    \item \textbf{Theoretical Guarantees:} We prove that our method is \emph{valid} (concept at the core of CP); we also prove that the optimal parameters  result in determining the empirical minimum average-radius conformal set for any pre-specified normed-ball region.
    
    \item \textbf{CP-Efficiency \& Computational Efficiency:} Numerical experiments on 4 case studies  show that OSCP produces valid conformal regions with the smallest size among baselines \cite{ stankeviciute2021conformal,sun2024copula, Tumu2024multi-modal, cleaveland2024conformal, gal2016dropout}. Compared to the previous State-of-the-Art (SOTA) method \cite{cleaveland2024conformal}, OSCP reduces the conformal set size by 16.03\%, 14.32\%, 14.01\%, 16.93\% on the 4 case studies, respectively; Our optimization program runtime is 8812.0, 78622.0, 14.4, 22.1 times faster on the 4 cases compared to \cite{cleaveland2024conformal}. 
    \item \textbf{Reproducibility:} To enhance reproducibility, the source code of the proposed algorithm is available online at \url{https://github.com/boyupang/OSCP}.
\end{enumerate}

The remainder of this paper is organized as follows:  Section \ref{sec: Problem setting} introduces the problem setting and the conformal prediction. In Section \ref{sec: OSCP} we formally propose our approach, which we term Optimal Selection Conformal Prediction (OSCP). Then Section \ref{sec: num_exp} compares our method with 5 baseline methods via numerical experiments on 3 synthetic datasets and 1 real dataset. Finally, Section \ref{sec: conclusion} concludes the study.

\section{Problem Setting and Conformal Prediction Preliminaries}\label{sec: Problem setting}
\subsection{Problem Setting}
In a discrete-time control system, let $\hat{\mathbf{Y}}_{0:T-1}=(\hat{Y}_0,...,\hat{Y}_{T-1}) \in \mathcal{Y} \subseteq \mathbb{R}^{d\times T}$ and $\mathbf{Y}_{0:T-1}=(Y_0,...,Y_{T-1}) \in \mathcal{Y} \subseteq \mathbb{R}^{d\times T}$ denote the nominal (predicted) and true trajectories of a quantity $Y_t$ (e.g., state, obstacle position) evolving over $T$ time-steps. For example, $\mathbf{Y}_{0:T-1}$ can be the trajectory of a moving obstacle to be avoided, while $\hat{\mathbf{Y}}_{0:T-1}$ is the predicted trajectory given by a neural-network. As another example, $\hat{\mathbf{Y}}_{0:T-1}$ can be the system state trajectory provided by the nominal system model which does not account for noise/disturbance, and thus different from the real trajectory $\mathbf{Y}_{0:T-1}$.
Let $\mathbf{\tilde{Y}}_{0:T-1}:=(\tilde{Y}_0,...,\tilde{Y}_{T-1})$ be the residual sequence capturing the error between the nominal and the true trajectory, i.e.,  $\tilde{Y}_t:=Y_t-\hat{Y}_t$.

We stipulate that the residual sequence $\mathbf{\tilde{Y}}_{0:T-1}$ is a random quantity distributed according to a probability measure $\mathbb{P}$. We assume that the corresponding probability space is defined as appropriate.

We assume throughout that we are given an \emph{exchangeable} calibration dataset $D_{\mathrm{cal}}=\{\mathbf{\tilde{Y}}_{0:T-1}^{(i)}\}_{i=1}^{N}$ containing the residual sequences of historical trajectories, where the term exchangeability is defined as follows: \begin{definition}[Exchangeability] \label{def: exchangeability}
    A collection of $N$ random variables is said to be  exchangeable if the joint probability distribution of any permutation of these $N$ random variables is the same.\footnote{Note that exchangeability is a weaker condition compared to assuming that data are independent and identically distributed (i.i.d.).}
\end{definition} It should be noted that we only assume  exchangeability at the $T$-horizon trajectory level, without requiring exchangeability within the time-horizon; i.e., we allow $\tilde{Y}_t$ and $\tilde{Y}_{t'}$ to be correlated for a given $0 \leq t\neq t' < T$. Another important remark is that if the nominal trajectory is generated from a data-driven model (e.g., neural-network), the calibration dataset $D_{\mathrm{cal}}$ must not involve any residual of training data, as we have assumed that all data come from the same distribution and such an operation would alter it.

For a pre-defined error level $\epsilon \in (0,1)$, our goal is to use $D_{\mathrm{cal}}$ to construct a 
set-valued predictor $\Gamma^\epsilon$ that predicts a closed and bounded abstraction region $\mathcal{C}_t$ (such as a norm-ball) around each $\hat{Y}_t$ such that with at least $(1-\epsilon)$ probability, the true trajectory $\mathbf{Y}_{0:T-1}$ is completely inside these abstraction regions simultaneously for each $t = 0,\ldots,T-1$. 

More formally, we want to use $D_{\mathrm{cal}}$ to construct a set-valued predictor
\begin{equation}
    \Gamma^\epsilon: \mathbf{\hat{Y}}_{0:T-1} \mapsto \otimes_{t=0}^{T-1} \mathcal{C}_t \subset \mathcal{Y}
\end{equation}
that produces $T$ decoupled and \emph{valid} (Def. \ref{def: validity}) abstraction regions for $\mathbf{Y}_{0:T-1}$, where the term validity is defined as follows.
\begin{definition}[Validity, \cite{vovk2005algorithmic}] \label{def: validity}
    Given a desired error level $\epsilon \in (0, 1)$, 
    a statistical abstraction predictor $\Gamma^\epsilon$ is said to be valid if for any new trajectory $ \mathbf{Y}^{(\mathrm{new})}_{0:T-1}$, we have \begin{equation}\label{eq: validity}
    {\mathbb{P}}\left(\mathbf{Y}^{(\mathrm{new})}_{0:T-1} \in \Gamma^\epsilon(\hat{\mathbf{Y}}^{(\mathrm{new})}_{0:T-1})\right) \geq 1-\epsilon ,
\end{equation}
\end{definition} 
where $\mathbb{P}$ is the joint probability measure  of the new residual sequence $\mathbf{\tilde{Y}}_{0:T-1}^{(\mathrm{new})}$ and the calibration data $D_{\mathrm{cal}}$. In the setting of this paper, \eqref{eq: validity} is equivalent to:
\begin{equation}
    {\mathbb{P}}\left(Y^{(\mathrm{new})}_t \in \mathcal{C}_t \text{ for all } t=0,...,T-1 \right) \geq 1-\epsilon.
\end{equation}
\subsection{Conformal Prediction and ICP Framework}\label{subsec: ICP intro}

\cite{vovk2005algorithmic} provides a computationally efficient framework, namely Inductive-Conformal-Prediction (ICP, also called Split-Conformal-Prediction) to satisfy \emph{validity}. Given a score function $\mathcal{A}: \mathbb{R}^{d \times T} \to \mathbb{R}$, we compute scores $R_i = \mathcal{A}(\mathbf{\tilde{Y}}_{0:T-1}^{(i)})$ for each $i \in \{1, \dots, N\}$. The $(1-\epsilon)$-quantile score, $R^{[p]}$, is defined as the $p$-th smallest value $R_i$, where $p = \lceil (1-\epsilon)(N+1) \rceil$.\footnote{Throughout the paper, for any calibration-set  size $m$, we assume $\epsilon \ge 1/(m+1)$, so that the index $\lceil (1-\epsilon)(m+1) \rceil \in \{1,...,m\}$ is well-defined.} Under the exchangeability assumption, any new trajectory with $\mathcal{A}(\mathbf{\tilde{Y}}_{0:T-1}^{(\mathrm{new})}) \leq R^{[p]}$ is guaranteed to be in the valid region.

The geometry of $\mathcal{C}_t$ is determined by $\mathcal{A}$. One crucial challenge is how to define the score function $\mathcal{A}$ such that the resulting CP-region is not conservative (too large). To evaluate the non-conservativeness of a CP method, we use the term \emph{CP-efficiency}: \begin{definition}[CP-Efficiency, \cite{vovk2005algorithmic}]\label{efficiency}
    Given an efficiency metric $\mathcal{L}_{\mathrm{eff}}$, an error level $\epsilon$, and a fixed input $\hat{\mathbf{Y}}_{0:T-1}$, a CP method $\Gamma^\epsilon_1$ is said to be more efficient than another CP method $\Gamma^\epsilon_2$ if \begin{equation}
\mathcal{L}_{\mathrm{eff}}\left(\Gamma^\epsilon_1(\hat{\mathbf{Y}}_{0:T-1})\right) < \mathcal{L}_{\mathrm{eff}}\left(\Gamma^\epsilon_2(\hat{\mathbf{Y}}_{0:T-1})\right). 
    \end{equation} \end{definition} In the context of time-series setting where we need to produce a sequence of $T$ decoupled regions, the efficiency metric $\mathcal{L}_{\mathrm{eff}}$ is usually taken as the sum-of-widths (diameters) or sum-of-volumes (Lebesgue measure) of the $T$ CP-regions.

To design efficient score functions, recent works \cite{sun2024copula, Tumu2024multi-modal, cleaveland2024conformal} adapt score parameters to calibration data via optimization or learning. To maintain exchangeability, $D_{\mathrm{cal}}$ is partitioned into two disjoint subsets: the first for parameter learning and the second for standard ICP calibration. This framework is termed \emph{Re-calibrated ICP}. In the following section, we follow this framework and propose an optimization-based score design.

\section{Optimal Selection Conformal Prediction (OSCP)}\label{sec: OSCP}
We introduce OSCP, an optimization-based method within the Re-calibrated ICP paradigm. By employing a novel additive-offset parameterization, OSCP minimizes the average radius of empirical CP regions using the first half of calibration dataset via a Mixed-Integer Linear Program (MILP), achieving superior CP-efficiency and computational performance compared to the state-of-the-art (SOTA) method.

\subsection{Motivation: Minimal-Average-Radius Regions Containing $p$ Residuals}\label{subsec: method motivation}
Standard ICP for a simple 2D regression task using an $\ell_2$-norm score function, constructing the valid $(1-\epsilon)$ CP region is equivalent to finding the smallest circle (centered at the origin) containing at least $p = \lceil(1-\epsilon)(N+1)\rceil$ residual vectors from the calibration set. Its radius corresponds to the $p$-th smallest score. When extending this to multi-dimensional time series, our goal is to construct a sequence of norm-balls over the horizon $T$ that tightly encapsulates $p$ residual trajectories. While explicitly minimizing the joint geometric volume of these sets is computationally expensive, minimizing the sum of their radii (the average radius) serves as a highly efficient and effective linear proxy. OSCP formalizes this by finding the minimal-average-radius norm-balls that cover the required fraction of calibration residuals.

\begin{figure}[t]
    \centering

    \begin{minipage}{0.4\linewidth}
        \vspace{1.15em}
       \centering
        \includegraphics[width=0.65\linewidth]{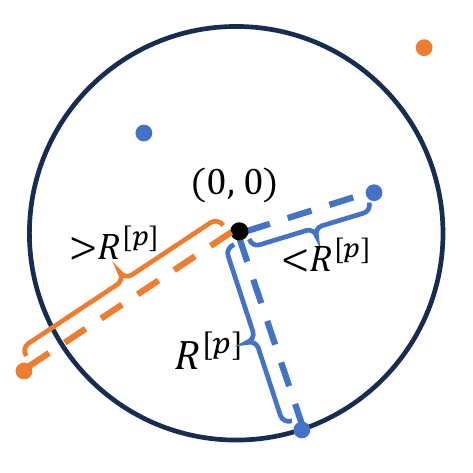}
        \vspace{1.1em}
        \subcaption{2D-regression example}
        \label{subfig: motivation_from_2d_regression}
    \end{minipage}
    \begin{minipage}{0.56\linewidth}
        \centering
        \includegraphics[width=1\linewidth]{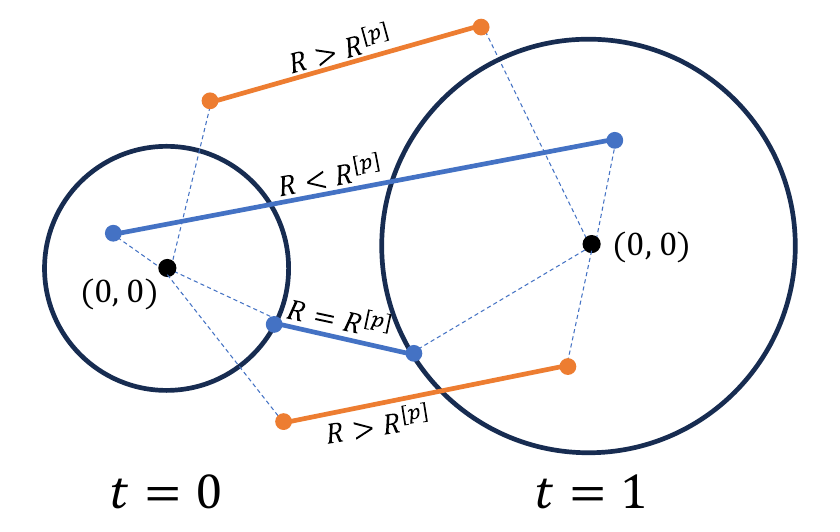}
        \subcaption{time-series data with $d=2, T=2$}
        \label{subfig: motivation_extend_to_time_series}
    \end{minipage}

    \caption{Data with non-conformity scores lower or equal to $R^{[p]}$ are drawn in blue, otherwise in orange. A valid CP contains at least $p$ residuals inside the CP region. In (b), each line segment is a residual sequence; it is contained inside the CP region if both of its ends are inside the circles.} 
    \label{fig:mainB}
\end{figure}

\subsection{OSCP: Algorithm Description, Validity and CP-Efficiency}\label{subsec: algo}
 Given a user-specified norm $||\cdot||: \mathbb{R}^d \rightarrow \mathbb{R}$, we first define the normed-residual-series.

\begin{definition}[Normed-residual-series $\tilde{\boldsymbol{e}}$]\label{def: residual-series}
    The normed-residual-series $\tilde{\boldsymbol{e}}^{(i)}_{0:T-1}=(\tilde{e}^{(i)}_0,...,\tilde{e}^{(i)}_{T-1})^\top$ for a residual sequence $\tilde{\mathbf{Y}}^{(i)}_{0:T-1}$ is defined as  
    \begin{equation}
        \tilde{\boldsymbol{e}}^{(i)}_{0:T-1}:= (||\tilde{Y}_0^{(i)}||, ..., ||\tilde{Y}_{T-1}^{(i)}||)^\top \in \mathbb{R}^T.
    \end{equation}
\end{definition} 

To provide a clear overview of the OSCP approach, we summarize the procedural flow as follows:

\begin{enumerate}
    \item[1)] \textbf{Split calibration data:} We randomly split the calibration dataset $D_{\mathrm{cal}}$ into two disjoint subsets $D_{\mathrm{cal},1}$, $D_{\mathrm{cal},2}$ with size $n_1$, $n_2$. Although there is no requirement on how to split the dataset, in our numerical implementation we use $n_1\approx n_2$. 
    \item[2)] \textbf{Optimize the parameters of score function:} Define the parameterized score function $\mathcal{A}$ by \begin{equation*}
     \mathcal{A}(\tilde{\mathbf{Y}}^{(i)}_{0:T-1}) := \max \left\{ \tilde{e}_0^{(i)} - r_0,..., \tilde{e}_{T-1}^{(i)} - r_{T-1} 
     \right\},
 \end{equation*} where $r_0,...,r_{T-1}$ are parameters that need to be determined. Specifically, we use the first set, $D_{\mathrm{cal},1}$, to formulate an MILP problem to find parameters $r_0^*,...,r_{T-1}^*$ (see Section \ref{subsec: Compute r_t}).

 \item[3)] \textbf{Score:} Once we have $r_0^*,...,r_{T-1}^*$, $\mathcal{A}$ is well-defined. Now we can calculate non-conformity scores $R_i=\mathcal{A}(\tilde{\mathbf{Y}}^{(i)}_{0:T-1})$ for each data point in the second set, $D_{\mathrm{cal},2}$.

 \item[4)] \textbf{Construct the CP region:} Let $p_2=\lceil (1-\epsilon)(n_2+1)\rceil$. Find the $p_2$-th smallest score $R^{[p_2]}$ from $D_{\mathrm{cal},2}$. Output the CP region  $\Gamma^\epsilon: \hat{\mathbf{Y}}_{0:T-1} \mapsto \otimes_{t=0}^{T-1} \mathcal{C}_t$, where the region at each time step is:  
\begin{equation*}\label{formulation of cp-regions}
    \mathcal{C}_t = \{y\in \mathbb{R}^{d} : ||y-\hat{Y}_t|| \leq R^{[p_2]}+r_t^*\}.
\end{equation*}
\end{enumerate}

\begin{theorem}[\emph{Validity}] \label{thm: valid probability guarantee}
    The CP-region $\Gamma^\epsilon$ produced by OSCP is valid, satisfying  trajectory-level coverage:
    ${\mathbb{P}}\left(Y^{(\mathrm{new})}_t \in \mathcal{C}_t, \forall t=0,...,T-1 \right) \geq 1-\epsilon,$ where $\mathbb{P}$ is the joint probability measure of $D_{\mathrm{cal},2}$ and $\tilde{\mathbf{Y}}_{0:T-1}^{(\mathrm{new})}$.
\end{theorem}
\begin{proof}
    Because parameters $r_t^*$ depend exclusively on $D_{\mathrm{cal},1}$, the scores computed on $D_{\mathrm{cal},2}$ remain exchangeable with the score of a new test point. Applying the marginal-coverage guarantee of standard ICP (Lemma 1 in \cite{Tibshirani2019Conformal}), the conformal regions $\Gamma^\epsilon(\hat{\mathbf{Y}}_{0:T-1}^{(\mathrm{new})}):=\left\{\mathbf{Y}\in \mathbb{R}^{d\times T} \mid \mathcal{A}(\mathbf{Y}-\hat{\mathbf{Y}}^{(\mathrm{new})}_{0:T-1}) \leq R^{[p_2]}\right\}$ have property that $
    {\mathbb{P}}\left(\mathbf{Y}^{(\mathrm{new})}_{0:T-1}\in \Gamma^\epsilon(\hat{\mathbf{Y}}^{(\mathrm{new})}_{0:T-1})\right) \geq 1-\epsilon$. Now, for any $\mathbf{Y} \in \mathbb{R}^{d\times T}$,
    
       $ \mathcal{A}(\mathbf{Y}-\hat{\mathbf{Y}}^{(\mathrm{new})}_{0:T-1}) =\underset{t}{\max}\left\{||Y_t-\hat{Y}_t^{(\mathrm{new})}||-r_t^*\right\} \leq R^{[p_2]}\\
        \Leftrightarrow ||Y_t-\hat{Y}_t^{(\mathrm{new})}|| \leq R^{[p_2]} + r_t^*, \quad \forall t=0,\ldots,T-1.$

    Thus, we can guarantee that the CP-region is valid.
\end{proof}

\paragraph{CP-Efficiency of OSCP} Within the class of norm-ball regions induced by the proposed score function, the learned parameters minimize the empirical average radius on $D_{\mathrm{cal},1}$. We formalize this in the theorem below.

\begin{theorem}[Empirical Average Radius Minimization] \label{thm: my method is least conservative}
    Suppose $\{r_0^*, \dots, r_{T-1}^*\}$ are the optimal parameters computed in Step 2 (see also the optimization program in Section \ref{subsec: Compute r_t}), based on $D_{\mathrm{cal},1}$.
The minimum average radius of an empirical CP-region on $D_{\mathrm{cal},1}$ is then equal to $\frac{1}{T}\sum_{t=0}^{T-1}r_t^*$. 
\end{theorem}

Note that the term ``empirical'' refers to the Empirical-Risk-Minimization (ERM) principle. The regions generated by calibrating data in $D_{\mathrm{cal},1}$ via score function $\mathcal{A}(\tilde{\mathbf{Y}}^{(i)}_{0:T-1}) := \max \left\{ \tilde{e}_0^{(i)} - r_0^*,..., \tilde{e}_{T-1}^{(i)} - r_{T-1}^* 
     \right\}$ are not \emph{valid} CP regions, but this Theorem shows that OSCP is \emph{efficient} in the sense of ERM principle. The proof can be found in Appendix \ref{Appendix-proof: proof of my method is least conservative}.

\paragraph{Shapes of CP regions}
 The pre-defined norm determines the shape of CP region. For instance, $\ell_2$-norm produces ball-shaped regions, $\ell_1$ or $\ell_\infty$-norm yields hyper-rectangle regions, and positive-definite matrix $A$-norm results in ellipsoid-shaped regions (see \cite{Xu2024Conformal}). Ellipsoidal regions are flexible, but they are more likely to overfit when there is insufficient data to reflect its shape in high dimensional spaces. When the dataset is small, $\ell_2$-norm is usually more stable.

\subsection{Optimal Parameter Computation}\label{subsec: Compute r_t} \label{Compute r_t}
Suppose we have selected the shape of CP region by choosing a specific norm $||\cdot||$, and we have calculated the normed-residual-series for the data in $D_{\mathrm{cal},1}$. Let $p_1=\lceil (1-\epsilon)(n_1+1)\rceil$. Our goal is to determine optimal parameters $r_0^*,...,r_{T-1}^*$ by using the first calibration dataset $D_{cal, 1}$.

Recalling the motivating example in Section \ref{subsec: method motivation}, we seek to determine a series of norm-balls with radii $r_t$, $t=0,\ldots,T-1$, that have the minimum average radius (equivalently radius sum, i.e., $\sum_{t=0}^{T-1}r_t$) and contain at least $p_1$ \emph{normed-residual-series} $\tilde{\boldsymbol{e}}^{(i)}_{0:T-1}$'s. We can achieve this by means of the following optimization problem: \begin{align}
   & \underset{\{r_t\},  \{b_i\}}{\min} && \sum_{t=0}^{T-1}r_t \tag{MILP} \label{opt-formulation:MILP} \\
   & \text{subject to} && r_t \ge   \tilde{e}_t^{(i)} b_i, \quad \forall t=0,\dots,T-1 \label{opt: line: radius constr} , \\ & && \qquad \qquad \qquad \qquad \ \ \ i=1,\dots,n_1, \nonumber \\
   & && \sum_{i=1}^{n_1}b_i = p_1, \label{opt: line: b constr} \\
   & && b_i \in \{0, 1\}, \quad  \forall i=1,..., n_1. \label{opt: line: integer constr}
\end{align}

This is a mixed-integer linear programming problem that is always feasible, and in fact we can remove a large fraction of redundant constraints to enable faster computation while keeping the optimal solutions of $r_0,...,r_{T-1}$ unchanged.
To reduce the size of this program and improve the associated computational efficiency, two redundant constraint sets can be identified and removed. To this end, suppose that for each t, $\tilde{e}_t^{[p_1]}$ is the $p_1$th smallest value among $\tilde{e}_t^{(1)},...,\tilde{e}_t^{(n_1)}$. Then the first redundant constraint set is defined by \begin{equation}
    S_1 := \left\{i \in \{1,\ldots,n_1\}\mid \tilde{e}_t^{(i)} \leq \tilde{e}_t^{[p_1]}, \ \forall t=0,...,T-1\right\}.
\end{equation} 
This set denotes the indices of all inactive constraints. In the case we are provided (or often it is easy to identify) a feasible solution $\{r_0^{(\mathrm{feas})},..., r_{T-1}^{(\mathrm{feas})}, b_1^{(\mathrm{feas})},...,b_{n_1}^{(\mathrm{feas})}\}$ to the \eqref{opt-formulation:MILP}, then we can neglect a second redundant set \begin{equation}
   S_2 := \left\{ i \in \{1,...,n_1\}\mid  \tilde{e}_t^{(i)} > r_t^{(\mathrm{feas})} \ \text{for} \ \forall t=0,...,T-1\right\},
\end{equation}
which includes all solutions that would lead to a cost (sum of radii) greater than that of the available feasible solution.

Although the method to find such feasible solutions is not unique, one fast and easy-to-implement heuristic procedure is as follows:
for each $i=1,...,n_1$, we calculate the sum of normed residuals $\mathrm{ResSum}(i):= \sum_{t=0}^{T-1} \tilde{e}^{(i)}_t$ and sort $\mathrm{ResSum}(i)$'s in non-decreasing order. Then we pick the first $p_1$ indices, $i_1,...,i_{p_1}$ and let $b_i^{(\mathrm{feas})}=1$ for these indices, $b_i^{(\mathrm{feas})}=0$ otherwise. 
Let $r_t^{(\mathrm{feas})}=\underset{i=i_1,...i_{p_1}}{\max}\tilde{e}_t^{(i)}$. Then we have a feasible solution to \eqref{opt-formulation:MILP}.

        

Once $S_1$ \& $S_2$ are identified, we can set up a modified optimization program as follows. \begin{align}
   & \underset{\{r_t\},  \{b_i\}}{\min} && \sum_{t=0}^{T-1}r_t \tag{MILP-fast} \label{opt-formulation:MILP-fast} \\
   & \text{subject to} && r_t \ge   \tilde{e}_t^{(i)} b_i, \quad \forall t=0,\dots,T-1 \label{opt-fast: line: radius constr} , \ i\in S, \\ 
   & &&  r_t\ge \underset{i\in S_1}{\max}\{\tilde{e}_t^{(i)}\}, \quad \forall t=0,\dots,T-1, \label{opt-fast: line: lb of rt}\\
   & && \sum_{i\in S}b_i = p_1 - |S_1|  \label{opt-fast: line: cstr for b} \\
   & && b_i \in \{0, 1\}, \quad \forall i\in S, \label{opt-fast: line: integer constr} 
\end{align}
\begin{equation*}
    \text{where } S := \{1,2,...,n_1\} \setminus (S_1\cup S_2).
\end{equation*}

The pseudo-code of this faster MILP (including the heuristic) is in Algorithm \ref{algo: Faster program for r_t}.

\begin{algorithm}[htbp]
    \caption{Faster MILP for computing parameters $r_t$}
    \begin{algorithmic}[1] 
    \label{algo: Faster program for r_t}
    \REQUIRE $\{\{\tilde{e}_t^{(i)}\}_{t=0}^{T-1}\}_{i=1}^{n_1}$, $p_1$
    \ENSURE CP parameters $r_0^*,..., r_{T-1}^*$  \\
    // Heuristic of finding feasible parameters $r_t$
    \FOR{$i=1,2,...,n_1$}
    \STATE $\mathrm{ResSum}(i) \leftarrow \sum_{t=0}^{T-1} \tilde{e}_{t}^{(i)}$
    \ENDFOR \\
    \STATE Sort $\mathrm{ResSum}(1),...,\mathrm{ResSum}(n_1)$ in non-decreasing order
    \STATE $i_1,...,i_{p_1} \leftarrow$ indices of first $p_1$ element in the sorted array
    \STATE let $r_t^{(\mathrm{feas})} = \underset{i=i_1,...,i_{p_1}}{\max}\tilde{e}_t^{(i)}$ \\
    \STATE $S_1 \gets \emptyset,\ S_2 \gets \emptyset$ \quad // Finding $S_1$ \& $S_2$
    \FOR{$i=1,2,...,n_1$}
        \IF {$|S_1|=p_1$} \label{line: corner case-if}
            \STATE Exit and output $r_t^* = \tilde{e}_t^{[p_1]}, \forall t$ \label{line: corner case-do}
        
        \ELSIF{$\tilde{e}_t^{(i)} \leq \tilde{e}_t^{[p_1]}$ for all $t=0,...,T-1$}
            \STATE add i into $S_1$
        \ELSIF{$\tilde{e}_t^{(i)} > r_t^{(\mathrm{feas})}$ for all $t=0,...,T-1$}
            \STATE add i into $S_2$
        \ENDIF        
    \ENDFOR
    \STATE Solve \eqref{opt-formulation:MILP-fast} \& output the optimal $\{r_0^*,..., r_{T-1}^*\}$
    \end{algorithmic}
\end{algorithm}

\begin{theorem}[Equivalence of \eqref{opt-formulation:MILP}  \& \eqref{opt-formulation:MILP-fast}] \label{thm: equivalence property for optimization-fast}
    When $|S_1| < p_1$, \eqref{opt-formulation:MILP-fast} is always feasible, and its set of optimal solutions of $\{r_0^*,...,r_{T-1}^*\}$ coincides with that of \eqref{opt-formulation:MILP}. 
    Otherwise, if $|S_1|\geq p_1$, the optimal solution to \eqref{opt-formulation:MILP} is $r_t^* = \tilde{e}_t^{[p_1]}, \ t=0,...,T-1$.
\end{theorem}

The corresponding proof is in Appendix \ref{Appendix-proof: proof of equivalence property of opt-fast&opt}. Theorem \ref{thm: equivalence property for optimization-fast} implies that when $|S_1| < p_1$, we can solve \eqref{opt-formulation:MILP-fast}
to find optimal parameters $r_0^*,..., r_{T-1}^*$. The rough idea is that to make the choice of $r_t$'s optimal, we must always consider containing residual-time-series from $S_1$ but not considering containing those from $S_2$. Thus, when $S_1 < p_1$, solving \eqref{opt-formulation:MILP} is equivalent to solve \eqref{opt-formulation:MILP-fast}.
On the other hand, when $|S_1| \geq p_1$ (although not common in practice), there are already more than $p_1$ residual-time-series to be contained inside the norm-balls, so we can simply choose $r_t = \tilde{e}_t^{[p_1]}$, $t=0,\ldots,T-1$, as the optimal parameters.

When the error level $\epsilon$ is small, \eqref{opt-formulation:MILP-fast} usually can remove a large number of mixed-integer constraints in \eqref{opt: line: radius constr} and integer variables $b_i$'s, which makes the computation much faster. The detailed results of increased running speed can be seen in Section \ref{subsec: results of num exp}.

\section{Numerical Experiments}\label{sec: num_exp}

To demonstrate the performance of our method, we test on both simulated and real time-series with different time horizons $T$, dimensions $d$, and calibration dataset sizes $N$, taken from \cite{sun2024copula}. We compare our method with 5 baseline uncertainty quantification (UQ) methods, and the results show that among all \emph{valid} alternatives, the \emph{CP-efficiency} of our proposed approach outperforms the state-of-the-art method on all case studies.

\paragraph{Baseline Approaches}
We selected MC-dropout \cite{gal2016dropout} and CF-RNN \cite{stankeviciute2021conformal} as the baseline approaches for Bayesian UQ method and CP for time series, respectively; as well as three recent CP methods for time series, namely, CopulaCPTS \cite{sun2024copula} and LCP \cite{cleaveland2024conformal} as parameter optimization methods, and CRD \cite{Tumu2024multi-modal} as a convex CP baseline. Since both CRD and our method allow users to specify shapes, we test hyper-rectangular \& ellipsoidal shapes of CRD (denoted as CRD-Rect \& CRD-Ell, respectively) and ball \& ellipsoidal shapes of our method (OSCP-$\ell_2$ \& OSCP-Ell).

\subsection{Synthetic Datasets}

\paragraph{Particle trajectory} According to \cite{sun2024copula}, the first two datasets are generated from the Interacting Particle System \cite{Kipf2018Neural}, where Gaussian noise with standard deviations $\sigma=0.01$ and $0.05$ are added to the dynamics of particle simulations in two datasets, respectively. For each case, data is split into 2000/500/500 for training, calibration, testing, respectively. A prediction model is trained to predict the future dynamics $\mathbf{Y} \in \mathbb{R}^{d \times T}$ given the past observations $\mathbf{X} \in \mathbb{R}^{d \times \tau}$ of the particle simulation with $\tau = 35, \ T=25, \ d=2$. Then all UQ methods are evaluated according to the procedures stated in Appendix \ref{Appendix: experiment details}. For UQ methods that require a further split of calibration dataset, the split ratio is set to be $0.5/0.5$ for $D_{\mathrm{cal},1}$ \& $D_{\mathrm{cal},2}$ except LCP. We note that the optimization program in LCP becomes computationally intractable for this split, so we adopt $0.1/0.9$ ratio for LCP.

\paragraph{Drone trajectory}
The drone trajectory dataset is generated from \cite{sakai2018pythonroboticspythoncodecollection} with added Gaussian noise of $\sigma=0.02$. The data-split is $600/200/200$ for training/calibration/testing. The prediction model forecasts a drone's future trajectory given its past observations ($\tau = 60, \ T=10, \ d=3$). After training the prediction model, all UQ methods are evaluated in a similar manner. For methods requiring a split of the calibration dataset, a split-ratio of $0.5/0.5$ is adopted (including LCP).

\subsection{Real Dataset: Covid-19 Daily Cases}
We also conduct a case study on the real dataset, UK Covid-19 Daily Cases. We use the preprocessed dataset from \cite{sun2024copula}.\footnote{The original Covid-19 dataset can be downloaded at https://coronavirus.data.gov.uk/} Each time-series sample in the preprocessed Covid-19 dataset corresponds to $150$-day daily cases from mid-September $2020$ to mid-February $2021$ in a region in the UK. Then, $500/160/80$ samples are used for training/calibration/testing. The prediction model takes $100$ days of data as input, and outputs the subsequent $50$ days, i.e., $\tau=100, T=50, d=1$. For UQ methods that require a further split of the calibration data, the split ratio was set to $0.5/0.5$.

\subsection{Numerical Results}\label{subsec: results of num exp}

\begin{figure}[t]
    \centering
    \includegraphics[width=1\linewidth]{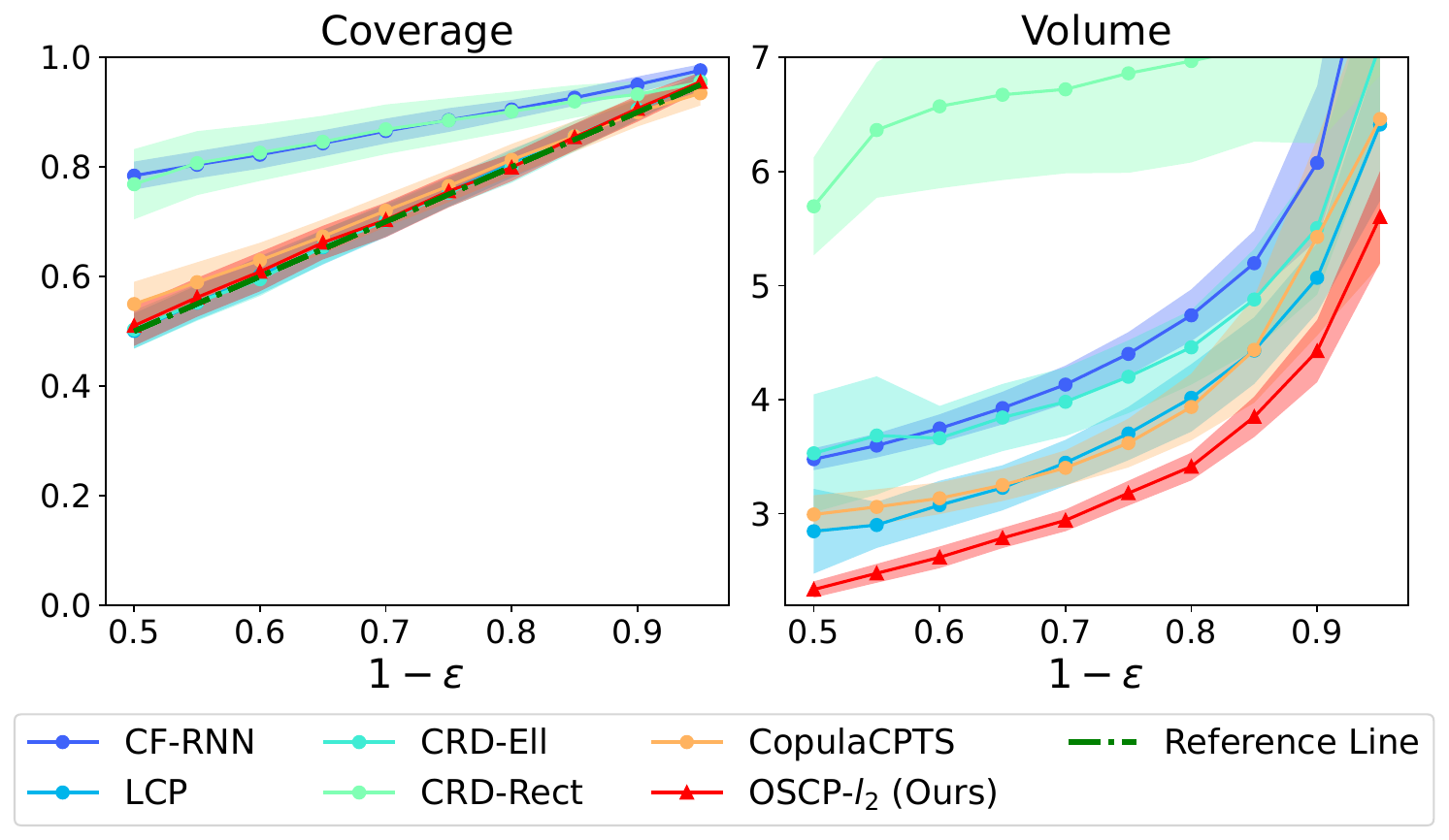}
    \caption{Performance visualization on the Particle Dataset ($\sigma=0.05$). The dashed reference line in the Coverage graph denotes the target confidences, and only methods with coverage curves at or above this line achieve the target coverages. In the Volume graph, curves closer to the bottom indicate better performance (less conservative).}
    \label{fig: cov_area_fig_for_maintext}
\end{figure}

For target confidence levels from 0.5 to 0.95 (10 values), we tested each UQ method with 50 runs (random splits of calibration and test set but with the same proportion) on each dataset. OSCP aims at producing minimal radius norm-balls; yet it is not straightforward to compare that radius with some UQ methods whose outcome is not a norm-ball one. Thus, we consider comparing the total volume (sum of per-step volume/area/length depending on the dimension) of confidence sets. Detailed descriptions of the comparison setup \&  evaluation metrics are  in the Appendix \ref{Appendix-exp: setup n eval details}.

Our method outperforms all the baselines on 4 case studies for all 10 confidence levels ($0.5$ to $ 0.95$). Specifically, compared to previous SOTA method LCP \cite{cleaveland2024conformal} (the one with the smallest total volume among baselines that have empirical coverage no smaller than target ones),  our method with $\ell_2$-norm, OSCP-$\ell_2$, reduces the total confidence region size (on average) by \textbf{16.03\%, 14.32\%, 14.01\%, 16.93\%} on dataset \textbf{Particle ($\sigma=0.01$)}, \textbf{Particle ($\sigma=0.05$)}, \textbf{Drone}, \textbf{Covid-19}, respectively. When using an ellipsoidal confidence region, our method achieves further size reductions (see Appendix \ref{Appendix-exp-details: visualize results}). 

Some of the experiment results (with standard error) are shown in Figure \ref{fig: cov_area_fig_for_maintext}, and the complete results \& visualizations are in Appendix \ref{Appendix-exp-details: visualize results}.  In Figure \ref{fig: cov_area_fig_for_maintext}, it can be seen that our method returns the smallest-volume confidence region among all alternatives, while achieving the target confidence levels. 


%

\paragraph{Runtime comparison} We also compare the runtime used in solving the optimization problem between our method and the previous SOTA (LCP, \cite{cleaveland2024conformal}). For the Particle ($\sigma=0.05$) dataset, LCP sometimes reaches the time limit and terminates the simulation at $10000$s, so the actual computing time is higher than the reported result. In Table \ref{Table: runtime comparison}, when target confidence is set as $0.9$, our method is \textbf{8812.0}, \textbf{78622.0}, \textbf{14.4}, \textbf{22.1} times faster than LCP on the 4 datasets, respectively. 
\begin{table}[htbp] 
  
\caption{Comparison of optimization runtime (in sec) for target confidence $1-\epsilon=0.9$}
\label{Table: runtime comparison}
  \centering 
    \resizebox{\linewidth}{!}{
    \begin{tabular}{ccc}
    \toprule
        Case Study & Previous SOTA \cite{cleaveland2024conformal} & OSCP-$\ell_2$  \\
        \midrule
        
        Particle ($\sigma=0.01$) & 1215.002 {\scriptsize $\pm$ 1450.119} & 0.137  {$\pm$ \scriptsize 0.057} \\
         
        Particle ($\sigma=0.05$)  & $>$9392.66 {\scriptsize $\pm$ 1358.109} & 0.119  {$\pm$ \scriptsize 0.034} \\
        
        Drone &  0.151 {\scriptsize $\pm$ 0.033} & 0.010 {\scriptsize $\pm$ 0.007} \\
        
        Covid-19 &  0.549 {\scriptsize $\pm$ 0.166} & 0.025 {\scriptsize $\pm$ 0.011} \\
    \bottomrule
    \end{tabular}}

\end{table}

\section{Conclusion}
\label{sec: conclusion}
In this work, we propose a new parameterized score function for conformal prediction in multi-dimensional time series and an optimization program to determine an optimal parameter set. We prove the validity and CP-efficiency of our method, showing that optimizing these parameters is equivalent to determining the minimum-average-radius CP regions with a pre-specified norm-ball description. Numerical results on four different datasets (synthetic and actual data) demonstrate that our method outperforms alternative approaches, while having much lower computational requirements.


\bibliography{ref}

\newpage
\section*{APPENDIX}
\section{Proof of Theorem \ref{thm: my method is least conservative}}\label{Appendix-proof: proof of my method is least conservative} \begin{lemma} \label{lemma: cp_region_lower_bound}
    Fix an error level $\epsilon \in (0,1)$, and a predefined score function $\mathcal{A}$. Fix also any norm, and let $r_t^*$, $t=0,\ldots,T-1$, be the optimal parameters of OSCP's score function. Any valid CP region (with same shape as OSCP's) constructed based on $D_{\mathrm{cal},1}$ has average radius $\frac{1}{T}\sum_{t=0}^{T-1}r_t \geq \frac{1}{T}\sum_{t=0}^{T-1}r_t^*$.
\end{lemma}

\begin{proof}[Proof of Lemma \ref{lemma: cp_region_lower_bound}] We first show that if the set (not necessarily a CP-region) $\Gamma^\epsilon(\hat{\mathbf{Y}}_{0:T-1}):={\otimes}^{T-1}_{t=0}\left\{y \in \mathbb{R}^{d} : ||y-\hat{Y}_t|| \leq r_t\right\}$ contains at least $p_1$ elements out of $\mathbf{Y}^{(i)}$, $i=1,\ldots,n_1$, then it has average radius no smaller than $\frac{1}{T}\sum_{t=0}^{T-1}r_t^*$.

    Let $p_1=\lceil (1-\epsilon)(n_1+1)\rceil$. Consider the set \begin{equation}\label{eq: optimal normed-ball region}
        \Gamma^\epsilon(\hat{\mathbf{Y}}_{0:T-1})^*:={\otimes}^{T-1}_{t=0}\left\{y \in \mathbb{R}^{d} : ||y-\hat{Y}_t|| \leq r_t^*\right\},
    \end{equation}
    where $\{r_t^*\}_{t=0}^{T-1}$ is the optimal solution to the \eqref{opt-formulation:MILP}. Since ${r_0^*, r_1^*, ..., r_{T-1}^*}$ is feasible to \eqref{opt-formulation:MILP}, there are at least
    $p_1$ indices from $\{1,...,n_1\}$ such that  $\tilde{e}^{(i)}_t \leq r_t^*, \forall t=0,1,...,T-1$. Now, since $\{r_t^*\}$ is optimal to the objective of \eqref{opt-formulation:MILP}, then the average radius of
    $\Gamma^\epsilon(\hat{\mathbf{Y}}_{0:T-1})^*$ is the minimum among all norm-ball regions that contains at least $p_1$ elements out of $\mathbf{Y}^{(i)}$, $i=1,\ldots,n_1$.

    Now, consider CP regions constructed on $D_{\mathrm{cal}, 1}$ via the score function $\mathcal{A}$ and a selected shape induced by $||\cdot||$:
\begin{equation}
    \Gamma^\epsilon(\hat{\mathbf{Y}}_{0:T-1}):=\{\mathbf{Y}\in \mathbb{R}^{T\times d}: \mathcal{A}(\mathbf{Y}-\hat{\mathbf{Y}}_{0:T-1})\leq R^{[p_1]}\}.  \nonumber
\end{equation}  
Since $R^{[p_1]}$ is the $p_1$th smallest nonconformity score, then there are at least $p_1$ number of i's satisfying 
\begin{equation}
    \mathcal{A}( \tilde{\mathbf{Y}}^{(i)}_{0:T-1}) \leq R^{[p_1]} 
    \Rightarrow \mathbf{Y}^{(i)}_{0:T-1}\in \Gamma^\epsilon(\hat{\mathbf{Y}}^{(i)}_{0:T-1}). \nonumber
\end{equation}
Then the average radius of $\Gamma^\epsilon(\hat{\mathbf{Y}}_{0:T-1})$ cannot be smaller than that of \eqref{eq: optimal normed-ball region}, which is $\frac{1}{T}\sum_{t=0}^{T-1}r_t^*$.
\end{proof}

\begin{proof}[Proof of Theorem \ref{thm: my method is least conservative}]
 Now with Lemma \ref{lemma: cp_region_lower_bound}, we start proving the Theorem \ref{thm: my method is least conservative}. Given an optimal solution $\{r_0^*,..., r_{T-1}^*, b_1^*,...,b_{n_1}^*\}$ of \eqref{opt-formulation:MILP}, the non-conformity score $R_i$ of each data point is then calculated by:
    \begin{equation}
        R_i:=\mathcal{A}( \tilde{\mathbf{Y}}^{(i)}_{0:T-1}) = \max\{\tilde{e}_0^{(i)}-r_0^*,...,\tilde{e}_{T-1}^{(i)}-r_{T-1}^*\}. \nonumber
    \end{equation}
    To prove the result stated in the theorem, we first show that $R^{[p_1]} = 0$. Let $i_1, i_2,..., i_{p_1}$ be the indices such that $b_i^*=1$. Due to feasibility of $r_t^*$ for any $t=0,..., T-1$, we have
                $\tilde{e}_t^{(i)} \leq r_t^*, \forall i=i_1,..., i_{p_1}$. Thus, for $i=i_1,..., i_{p_1}$, we have $R_i \leq 0$. We then have 
                $R^{[p_1]} \leq \underset{i=i_1,...,i_{p_1}}{\max} R_i \leq 0$.
        Now suppose for the sake of contradiction that $R^{[p_1]} < 0$. Then $\exists i_1', i_2',..., i_{p_1}'$ such that $R_i < 0, \forall i=i_1', i_2',..., i_{p_1}'$. 
        Consequently, we have $\tilde{e}_t^{(i)} < r_t^*, \forall t=0,...,T-1, \forall i=i_1',...,i_{p_1}'$. 
        
        Consider a new solution candidate
        \begin{equation}
            r_t' = \underset{i=i_1',...,i_{p_1}'}{\max}\tilde{e}_t^{(i)}, \quad b_i' = \begin{cases}
                1, & \text{if } i=i_1',...,i_{p_1}'; \\ 
                0, & \text{otherwise}.
                \end{cases}\nonumber
        \end{equation}
            
        It is easy to check that this new solution is feasible to \eqref{opt-formulation:MILP} and $r_t' < r_t^*, \forall t=0,...,T-1$. This contradicts to the fact
        that $\sum_{t=0}^{T-1}r_t^*$ is the minimum cost solution.        

        Thus, we can conclude that $R^{[p_1]} = 0$. Then for each t, the resulting CP-region is 
        \begin{align}
            \{y \in \mathbb{R}^{d} &: ||y-\hat{Y}_t|| \leq  R^{[p_1]}+r_t^*\}
           \nonumber \\
           &\Longleftrightarrow \{y \in \mathbb{R}^{d} : ||y-\hat{Y}_t|| \leq r_t^*\}. \nonumber
        \end{align}
        This completes the proof that the average radius of empirical CP-region of OSCP calculated from $D_{\mathrm{cal}, 1}$ is equal to $\frac{1}{T}\sum_{t=0}^{T-1}r_t^*$, which is the minimum value that a valid CP-region with same shape can attain by Lemma \ref{lemma: cp_region_lower_bound}.
\end{proof}

\section{Proof of Theorem \ref{thm: equivalence property for optimization-fast}}

\label{Appendix-proof: proof of equivalence property of opt-fast&opt}
\begin{proof} \textbf{Case 1}: $|S_1| < p_1$

Feasibility of \eqref{opt-formulation:MILP-fast} is easy to check, as we can always randomly pick $p_1$ indices $i_1,...,i_{p_1} \in S$, and then the solution $\{r_t = \underset{i}{\max} \ \tilde{e}_t^{(i)}\}_{t=0}^{T-1}$, $b_i = \begin{cases} 1, & \forall i=i_1,...,i_{p_1} \\ 0, & \mathrm{otherwise} \end{cases}$ is trivially
feasible to the problem. We will now prove each direction seperately.

\begin{itemize}
    \item[($\Rightarrow$)] w.t.s. Any optimal parameters $\mathbf{r}^*=(r_0^*,...,r_{T-1}^*)$ of \eqref{opt-formulation:MILP}, $\exists \mathbf{b}'=\{b_i', \ i\in S\}$ s.t. $(\mathbf{r}^*,\mathbf{b}')$ is an optimal solution to \eqref{opt-formulation:MILP-fast}.
\end{itemize}

We will first show that the feasibility region of \eqref{opt-formulation:MILP-fast} is a subset of that of \eqref{opt-formulation:MILP}.

First of all, we augment the space of decision variables $(\mathbf{r}, \mathbf{b})=\{r_0,...,r_{T-1}\} \cup \{b_i\}_{i\in S}$ to $(\mathbf{r}, \mathbf{\Bar{b}})=\{r_0,...,r_{T-1}, b_1,...,b_{n_1}\}$, i.e.
the feasibility region of \eqref{opt-formulation:MILP-fast} now becomes \eqref{opt: line: integer constr}, \eqref{opt-fast: line: radius constr}, \eqref{opt-fast: line: lb of rt} \& \eqref{opt-fast: line: cstr for b}. 

Consider following constraints:
\begin{align} 
    & b_i = 1, \quad i\in S_1\label{extra-constr: S1};\\
    & b_i = 0, \quad i\in S_2\label{extra-constr: S2}.
\end{align}
We add these two constraints on $(\mathbf{r}, \mathbf{\Bar{b}})$, which has no effect on the feasible region of original decision variables $(\mathbf{r}, \mathbf{b})$. That being saying, the feasibility region of \eqref{opt-formulation:MILP} is equivalent with that of augmented \eqref{opt-formulation:MILP-fast}, i.e.,
\begin{equation} \label{property: equivalent constraints}
    \eqref{opt: line: radius constr}, \eqref{opt: line: b constr}, \eqref{opt: line: integer constr}, \eqref{extra-constr: S1}, \eqref{extra-constr: S2} \Leftrightarrow \eqref{opt: line: integer constr}, \eqref{opt-fast: line: radius constr}, \eqref{opt-fast: line: lb of rt}, \eqref{opt-fast: line: cstr for b}, \eqref{extra-constr: S1}, \eqref{extra-constr: S2}.
\end{equation}
This result is easy to check: for any solution $r_t, b_i$ satisfying \eqref{extra-constr: S1} \& \eqref{extra-constr: S2},
\begin{align}
    &r_t \ge \tilde{e}_t^{(i)} \cdot b_i, \ t=0,\dots,T-1 , \ i=1,\dots,n_1\nonumber  \\
    &\Leftrightarrow r_t \ge \tilde{e}_t^{(i)} \cdot b_i, \nonumber 
    \ t=0,\dots,T-1 , \ i\in \{1,\dots,n_1\} \setminus S_2\nonumber  \\
       &\Leftrightarrow \begin{cases} r_t \ge \tilde{e}_t^{(i)} \cdot b_i, \ t=0,\dots,T-1 , \ i\in S,\nonumber  \\ 
             r_t \ge\underset{i\in S_1}{\max}\{\tilde{e}_t^{(i)}\}, \ t=0,\dots,T-1. \end{cases} \nonumber
\end{align}
Therefore, under \eqref{extra-constr: S1} \& \eqref{extra-constr: S2}, Constraint \eqref{opt: line: radius constr} $\Leftrightarrow$ Constraints \eqref{opt-fast: line: radius constr}, \eqref{opt-fast: line: lb of rt}.
Also, under \eqref{extra-constr: S1} \& \eqref{extra-constr: S2}, 
\begin{equation}
    \sum_{i=1}^{n_1}b_i = p_1 \Leftrightarrow \sum_{i\in S}b_i = p_1-|S_1|, \nonumber
\end{equation}
hence Constraint \eqref{opt: line: b constr} $\Leftrightarrow$ Constraint \eqref{opt-fast: line: cstr for b}, and thus, we can conclude the result in \eqref{property: equivalent constraints}.
Consequently, we have shown that the feasible region of augmented \eqref{opt-formulation:MILP-fast} is a subset of that of \eqref{opt-formulation:MILP}, which means the optimal objective value of \eqref{opt-formulation:MILP-fast} is no smaller than that of \eqref{opt-formulation:MILP}.

Now we will show that the optimal objective are indeed equal by showing that any set of optimal parameters $\{r_0^*,...,r_{T-1}^*\}$ of \eqref{opt-formulation:MILP} is feasible in \eqref{opt-formulation:MILP-fast}. 

Suppose $(\mathbf{r}^*,\mathbf{b}^*)=(r_0^*,...,r_{T-1}^*,b_1^*,...,b_{n_1}^*)$ is an optimal solution of \eqref{opt-formulation:MILP}.

First, we show that $b_i^*=0$ for $i\in S_2$.
For $\forall i \in S_2$, $ \tilde{e}_t^{(i)} > r_t^{(\mathrm{feas})} \text{for} \ \forall t=0,...,T-1$. Since $r_t^*$ is optimal, 
$\sum_{t=0}^{T-1}r_t^* \leq \sum_{t=0}^{T-1}r_t^{(\mathrm{feas})}$. Then there must be at least one $\hat{t}$ s.t. $r_{\hat{t}}^* \leq r_{\hat{t}}^{(\mathrm{feas})}$. 
This means that $\tilde{e}_{\hat{t}}^{(i)} > r_{\hat{t}}^{(\mathrm{feas})} \geq r_{\hat{t}}^*  $
Thus, to satisfy the constraint \eqref{opt: line: radius constr} in \eqref{opt-formulation:MILP}, it must be that $b_i^*=0$ for $\forall i\in S_2$.

Next, we will show that $\exists \{b_i'\}_{i\in S}$ s.t. $(\mathbf{r}^*,\mathbf{b}')$ is feasible to \eqref{opt-formulation:MILP-fast}. Let $i_1,..., i_{p_1}$ be the indices s.t. $b_i^*=1$. Then by the above result we know that $i_1,...,i_{p_1} \in \{1,...,n_1\}\setminus S_2=S\cup S_1$.
Select $p_1-|S_1|$ indices from $\{i_1,...,i_{p_1}\}\setminus S_1$ (this is always possible as $|S_1| < p_1$) and set $b_i'=1$ for these indices, and set $b_i'=0$ for the rest of the indices in $S$.

Then we have $p_1-|S_1|$ indices i's s.t. $i \in S$ and $\tilde{e}_t^{(i)} \leq r_t^*$ at each t. So constraint \eqref{opt-fast: line: radius constr} \& \eqref{opt-fast: line: cstr for b} are satisfied.
Now, since for $i_1,..., i_{p_1}$, $b_i^*=1$, it means that we have $p_1$ indices i's s.t. $\tilde{e}_t^{(i)} \leq r_t^*$ at each t.
Recall that $\tilde{e}_t^{[p_1]}$ is the $p_1$th smallest element in the sorted non-descending sequence $\{\tilde{e}_t^{(i)}\}_{i=1}^{n_1}$. This means that $
    \tilde{e}_t^{[p_1]} \leq r_t^*, \quad \forall t=0,...,T-1.$ Consequently, for $\forall i \in S_1$, \begin{align}
     & \tilde{e}_t^{(i)} \leq \tilde{e}_t^{[p_1]} \leq r_t^*, \ \forall t=0,...,T-1 \nonumber \ \Leftrightarrow \underset{i\in S_1}{\max}\{\tilde{e}_t^{(i)}\} \leq r_t^*, \nonumber
\end{align}
so constraint \eqref{opt-fast: line: lb of rt} is also satisfied and we can therefore conclude that 
$(\mathbf{r}^*,\mathbf{b}')$ is feasible to \eqref{opt-formulation:MILP-fast}. As optimal value of \eqref{opt-formulation:MILP} is always 
larger or equal to that of \eqref{opt-formulation:MILP-fast}, we can conclude that $(\mathbf{r}^*,\mathbf{b}')$ is optimal solution to 
\eqref{opt-formulation:MILP-fast} and the optimal value of \eqref{opt-formulation:MILP} and \eqref{opt-formulation:MILP-fast} are indeed equal.

\begin{itemize}
    \item[($\Leftarrow$)] w.t.s. For any optimal parameters $\mathbf{r}^*=(r_0^*,...,r_{T-1}^*)$ of \eqref{opt-formulation:MILP-fast}, $\exists \mathbf{b}'=(b_1',b_2',...,b_{n_1}')$ s.t. $(\mathbf{r}^*,\mathbf{b}')$ is an optimal solution to \eqref{opt-formulation:MILP}.
\end{itemize}

Suppose $(\mathbf{r}^*, \mathbf{b}^*)$ is an optimal solution of \eqref{opt-formulation:MILP-fast}. Consider solution $(\mathbf{r}^*, \mathbf{b}')=(r_0^*,...,r_{T-1}^*,b_1',...,b_{n_1}')$, where
\begin{equation}
    b_i'=\begin{cases}
    b_i^*, & i\in S; \\
    1, & i\in S_1; \\
    0, & i\in S_2. \nonumber
\end{cases}
\end{equation}
Then we have \begin{equation}
        \sum_{i=1}^{n_1}b_i' = \sum_{i\in S}b_i^* + |S_1|
        \overset{\text{Constraint } \eqref{opt-fast: line: cstr for b}}{=} p_1-|S_1| + |S_1| = p_1, \nonumber
    \end{equation}
so constraint \eqref{opt: line: b constr} is satisfied.
For constraint \eqref{opt: line: radius constr}, let's first consider the case $i \in S\cup S_2 $, the constraints $ \tilde{e}_t^{(i)} \cdot b_i' \le r_t^*, \ t=0,\dots,T-1$ are trivially satisfied. For case of $i\in S_1$, constraint \eqref{opt-fast: line: lb of rt} in \eqref{opt-formulation:MILP-fast} says that $\tilde{e}_t^{(i)} \cdot b_i'-r_t^*=\tilde{e}_t^{(i)}-r_t^*\leq 0, \  t=0,...,T-1$. 
Combining these results, constraint \eqref{opt: line: radius constr} is satisfied. As a result, the solution $(\mathbf{r}^*,\mathbf{b}')$ is feasible to \eqref{opt-formulation:MILP}.
Since in the proof of $(\Rightarrow)$ we have already shown that the optimal value of \eqref{opt-formulation:MILP-fast} is equal to \eqref{opt-formulation:MILP}, we can conclude that $(\mathbf{r}^*,\mathbf{b}')$ is optimal to \eqref{opt-formulation:MILP}.

\textbf{Case 2}: $|S_1|\geq p_1$

Consider the solution candidate $r_t^* = \tilde{e}_t^{[p_1]}, t=0,...T-1$. We will first show it is feasible to \eqref{opt-formulation:MILP}. Pick arbitrary $p_1$ indices $i_1, i_2,..., i_{p_1}$ from $S_1$, 
then let $b_i=1, \forall i=i_1,..., i_{p_1}$ and let $b_i=0$ otherwise. This guarantees constraints \eqref{opt: line: b constr} \& \eqref{opt: line: integer constr} satisfied. 
Since $\forall i\in S_1$, $\tilde{e}_t^{(i)} \leq \tilde{e}_t^{[p_1]} $ for $ \forall t=0,...,T-1$, constraint \eqref{opt: line: radius constr} 
is also satisfied. Thus, the solution $r_t^* = \tilde{e}_t^{[p_1]}, t=0,...,T-1$ is feasible to \eqref{opt-formulation:MILP}. 

Now, we will show that it is also optimal to \eqref{opt-formulation:MILP}. Suppose, for contradiction, $\exists$ a feasible solution
$\{r_0',..., r_{T-1}'\}$ of \eqref{opt-formulation:MILP} such that the objective value $\sum_{t=0}^{T-1}r_t' < \sum_{t=0}^{T-1}r_t^*$.
Then $\exists r_{\hat{t}}' < r_{\hat{t}}^* = \tilde{e}_{\hat{t}}^{[p_1]}$ for some $\hat{t}$. This means there are less than $p_1$ i's s.t. 
$\tilde{e}_{\hat{t}}^{(i)} \leq r_{\hat{t}}'$. However, constraints \eqref{opt: line: radius constr} \& 
\eqref{opt: line: b constr} together imply that there $\exists$ at least $p_1$ i's such that $\tilde{e}_{\hat{t}}^{(i)} \leq r_{\hat{t}}', \forall t=0,...,T-1$. Contradiction! Thus, the solution
$r_t^* = \tilde{e}_t^{[p_1]}, t=0,...,T-1$ is optimal to \eqref{opt-formulation:MILP}.
\end{proof}

\section{Additional Details of Numerical Experiments}
\label{Appendix: experiment details}
This section presents additional details of the numerical experiments in Section \ref{sec: num_exp}.

\subsection{Additional Information of Experiment Setup and Evaluation Metrics}\label{Appendix-exp: setup n eval details}
\paragraph{Experiment Setup}
All experiments are conducted on Windows 11 (64-bit) machine with i9-13900HX CPU (24 physical cores), 32GB RAM. The computations (including prediction models training) use CPU only. For solving the optimization problems in our method and in \cite{cleaveland2024conformal}, we use the commercial solver Gurobi Optimizer (version 12.0.0 build v12.0.0rc1). 

\paragraph{Performance Evaluation} We mainly focus on testing \emph{validity} and \emph{CP-efficiency} for the UQ methods. In each case study, the dataset is first split into two halves. The first half is used for training the time-series prediction model $f_{\mathrm{pred}}$. To eliminate the variance of performance caused by model's variance, the model is fixed once trained. Then we repeat the following procedures for 50 runs: \begin{enumerate}
    \item Randomly split the second half into calibration-set and test-set with fixed proportion 
    \item For error tolerance ($\epsilon$) from 0.05 to 0.5 (10 different values in total), use calibration-set to train the UQ methods. 
    \item Test the UQ methods on test-set and calculate the \emph{validity} \& \emph{CP-efficiency} metrics (see below) on test-set for $\epsilon=0.05$ to $0.5$.
\end{enumerate} 
Lastly, all the results are averaged over the 50 runs. 

\paragraph{Validity Metric}
At each run, the validity is evaluated by the empirical coverage on the test-set, which is calculated as follows: 
\begin{equation}
    \text{Coverage}(\Gamma^\epsilon) := \frac{1}{|D_{test}|} \sum_{(\mathbf{X}^{(i)}, \mathbf{Y}^{(i)})\in D_{test}} \mathbb{I}\left(\mathbf{Y}^{(i)} \in \Gamma^\epsilon(\mathbf{X}^{(i)})\right).
\end{equation}
\paragraph{CP-Efficiency Metric}
Since the width of the confidence set is difficult to calculate for some UQ methods, we instead consider calculating the total volume (Lebesgue measure) of the confidence sets as \emph{CP-efficiency} metric $\mathcal{L}_{\mathrm{eff}}$ of UQ methods. For $d=1$ \& $d=2$ cases, the "volume" corresponds to length and area, respectively.
For an error tolerance $\epsilon$, the total volume of a UQ method is computed as follows:
\begin{equation}
    \text{Volume}(\Gamma^\epsilon) := \frac{1}{|D_{test}|} \sum_{\mathbf{X}^{(i)}\in D_{test}}Volume\left(\Gamma^\epsilon(\mathbf{X}^{(i)})\right).
\end{equation}

\paragraph{Training Details of time-series Forecasting Model $f_{\mathrm{pred}}$} Since we use the similar experiment settings (e.g., datasets, evaluation metrics, etc.) as in \cite{sun2024copula}, we also use the experiment code of \cite{sun2024copula} to train the underlying time-series prediction model for all 4 case studies.\footnote{Experiment code of \cite{sun2024copula} can be download at https://github.com/Rose-STL-Lab/CopulaCPTS} Here is a brief description of the forecasting model $f_{\mathrm{pred}}$ for 4 case studies. 

The forecasting model for \textbf{Particle Datasets ($\sigma=0.01$ $\&$ $\sigma=0.05$)} is an RNN network with embedding size $=$ 24, where the hidden state is then passed through a linear network to concurrently predict the time-steps. Then the model is trained for 150 epochs and batch size $=$ 150. 

The forecasting model for the \textbf{Drone Dataset} is an RNN with embedding size $=$ 128 that is trained with 500 epochs and batch size $=$ 150.

For the \textbf{Covid-19 Dataset}, the forecasting model is an RNN with embedding size$=$128, and is trained with epochs $ = $ 200 \& batch size $ = $ 50.

\subsection{Detailed Experiment Results and Visualizations}\label{Appendix-exp-details: visualize results}
The complete results of our method with $\ell_2$-norm (OSCP-$\ell_2$) \& with ellipsoidal norm (OSCP-Ell) and of all 5 baselines (LCP \cite{cleaveland2024conformal}, CopulaCPTS \cite{sun2024copula}, CRD \cite{Tumu2024multi-modal}, CF-RNN \cite{stankeviciute2021conformal}, MC-dropout \cite{gal2016dropout}) on all 4 case studies are shown in Table \ref{Table: Performance Comparison}, Figure \ref{fig: particle datasets} \& Figure \ref{fig: drone n covid datasets}. 

In Table \ref{Table: Performance Comparison}, the data in \textbf{bold} have better CP-efficiency (smaller total volume) than previous SOTA (LCP, \cite{cleaveland2024conformal}) and have coverages higher or equal to the target confidences, while the data in \textcolor{gray}{gray} mean the empirical coverages are lower than the target confidences. It is important to note that if the empirical coverage of a method is slightly lower than the target confidence level (e.g. difference of coverages $\leq 2\%$), it does NOT necessarily mean that this method is \emph{invalid}. The small gap in coverage can also be caused by numerical inaccuracy in computation (e.g., matrix inverse, eigen-value computations), or simply the randomness of empirical coverage (which converges to $\lceil(1-\epsilon)\cdot(N_{cal,2}+1)\rceil/(N_{cal,2}+1)$ in probability as $|D_{test}|\rightarrow \infty$, by the weak law of large numbers). However, to ensure a fair comparison of \emph{CP-efficiency}, here we only compare the results whose empirical coverages are no smaller than the target ones.

\subsection{Further Discussions on the Experiment Results}
For deep UQ methods, MC-dropout \cite{gal2016dropout} does not have \emph{validity} guarantees and is thus far below the reference target levels for Particle ($\sigma=0.05$), Drone, and Covid-19 datasets; CF-RNN \cite{stankeviciute2021conformal} is one of the classic works of CP in time series which holds the \emph{validity} guarantee, but it offers overly large confidence regions. For the baselines from newer works, CRD \cite{Tumu2024multi-modal} is unstable and produces excessive large confidence sets that cannot be shown in Figure \ref{fig: drone n covid datasets}. We hypothesize that this is due to the error accumulation during the fitting and estimation procedures of the CRD method. The performance of CopulaCPTS \cite{sun2024copula} is also unstable. It has large variances for some cases (see the area graph with confidence = 85\%, 90\% \&  95\% in Figure \ref{fig: particle datasets}), and also has unstable coverages (see results on Covid-19 dataset). Among all baselines, LCP \cite{cleaveland2024conformal} has the best performance as it achieves target coverages while having smaller volumes than other baselines. However, the computation cost for LCP can be very large (see Table \ref{Table: runtime comparison} in Section \ref{subsec: results of num exp}), and it still has larger volumes than our method, OSCP. 

We can also observe that the size reduction is most significant on the Covid-19 dataset. This dataset has $d=1$, which means the total volume in this case is equivalent as the total width. As our method OSCP is originally built to produce minimal-total-width regions rather than minimal-total-volume ones, it explains why the volume reduction is highest on the \textbf{Covid-19} dataset. 

A general conclusion from the results in Table \ref{Table: Performance Comparison} and Figure \ref{fig: particle datasets} \& \ref{fig: drone n covid datasets} is that our method with $\ell_2$-norm, OSCP-$\ell_2$, outperforms all the baselines on 4 case studies for all 10 confidence levels ($0.5$ to $ 0.95$). Specifically, OSCP-$\ell_2$ has coverages higher than all the target confidence levels, and it achieves volume reductions (average value over 10 confidence levels) of \textbf{16.03\%}, \textbf{14.32\%}, \textbf{14.01\%}, \textbf{16.93\%} compared to previous SOTA (LCP, \cite{cleaveland2024conformal}) on \textbf{Particle ($\mathbf{\sigma=0.01}$)}, \textbf{Particle ($\mathbf{\sigma=0.05}$)}, \textbf{Drone}, \textbf{Covid-19} datasets, respectively.

Besides, our method with ellipsoidal norm, OSCP-$Ell$, achieves volume reduction of \textbf{13.65\%}, \textbf{14.81\%}, \textbf{33.5\%} compared to previous SOTA (LCP, \cite{cleaveland2024conformal}) on \textbf{Particle ($\mathbf{\sigma=0.01}$)}, \textbf{Particle ($\mathbf{\sigma=0.05}$)}, \textbf{Drone}, respectively (since there is no ellipsoid for 1d case, we did not run OSCP-$Ell$ on the Covid-19 dataset). Ellipsoidal norm fitted an ellipsoid shape for each time step using $D_{\mathrm{cal,1}}$, and ellipsoid norm performs better when the residuals exhibit greater non-Gaussianity and/or when the dimension $d$ is larger. This explains why we can observe a more significant volume reduction on Drone dataset ($d=3$) using OSCP-$Ell$ (33.5\%) compared to OSCP-$\ell_2$ (14.1\%), while we cannot see significant differences on the Particle datasets (13.65\% vs. 16.03\% and 14.81\% vs. 14.32\%, respectively). In a word, when $|D_{\mathrm{cal}}|$ is small and we have no prior knowledge of the behavior of the residual at each time step, we can just assume the residual at each time step is Gaussian-error and use OSCP-$\ell_2$. Otherwise, if we know that the residual is highly non-Gaussian and the $|D_{\mathrm{cal}}|$ is not too small, using OSCP-$Ell$ is more preferable.

\subsection{Construction of Ellipsoid-Norm}
Given initial observations $\mathbf{X}$ of a time series, suppose $\mathbf{Y} = (Y_1,...,Y_T) \in \mathbb{R}^{d\times T}$ are the true future values, and $\hat{\mathbf{Y}} = (\hat{Y}_1,...,\hat{Y}_T) \in \mathbb{R}^{d\times T}$ are the predicted values from $f_{\mathrm{{pred}}}$. 
For each time step $t$, the Ellipsoid-norm of residual $Y_t-\hat{Y}_t$ is defined by
\begin{equation}
    \tilde{e}_t = ||Y_t-\hat{Y}_t||:= \sqrt{(Y_t-\hat{Y}_t)^\top\hat{\Sigma}^{-1}(Y_t-\hat{Y}_t)},
\end{equation}
where $\hat{\Sigma}^{-1}$ is the inverse of empirical covariance matrix estimated from $Y_t^{(i)}-\hat{Y}_t^{(i)}$ for $i =1,...,n_1$ (i.e., data from $D_{\mathrm{cal,1}}$). 

We wish to highlight that the ellipsoid norm is fitted for each time step individually, so we have $T$ different ellipsoidal norms in total, which ensures a locally-adaptive shape fitting. To guarantee robust construction of ellipsoid-norms against numerical issues, one can use pseudo-inverses of covariance matrices via the singular-value-decomposition technique instead of the standard inverses (see \cite{Xu2024Conformal}). 

\begin{figure*}[htbp]
    \centering
    \includegraphics[width=1\linewidth]{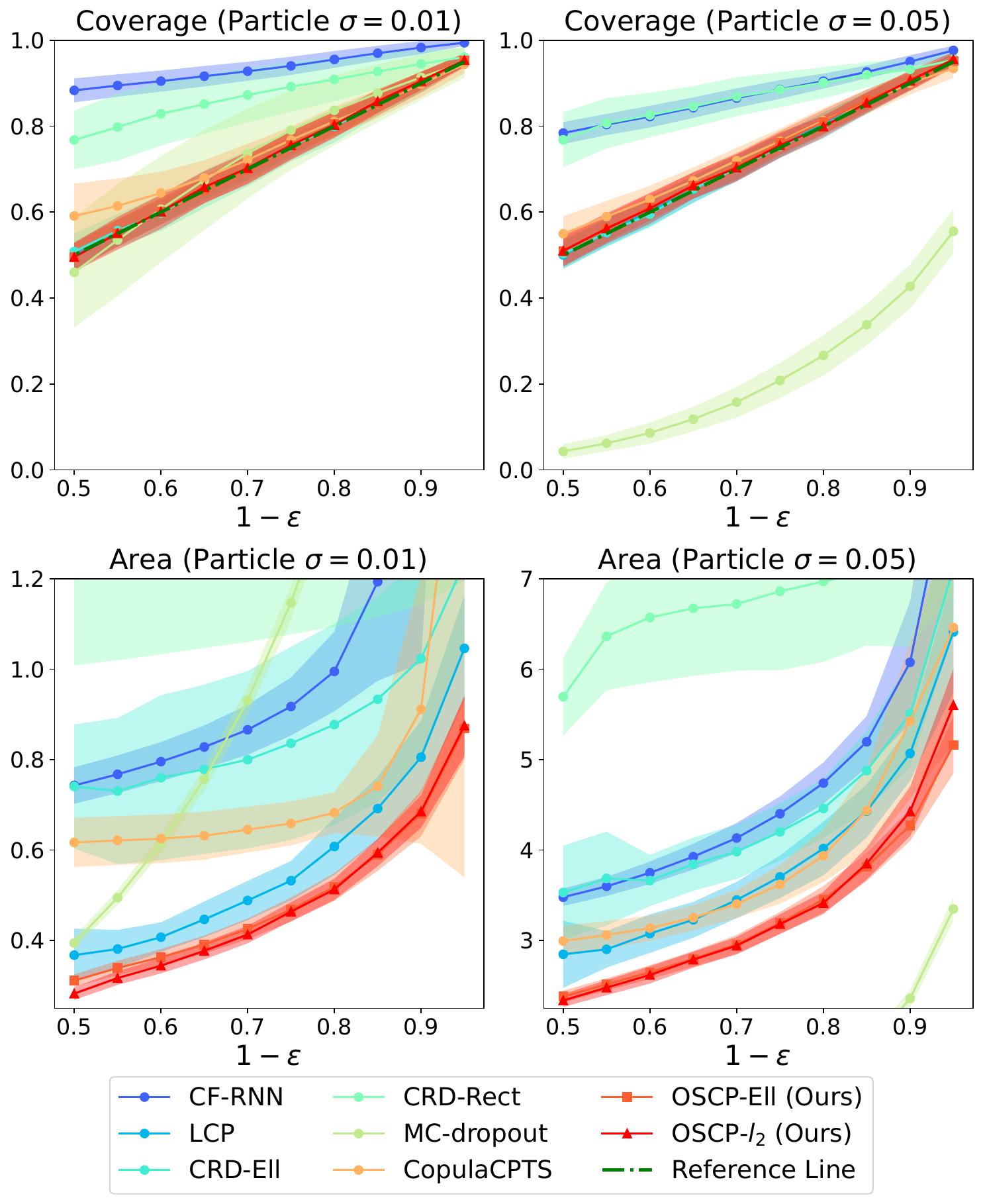}
    \caption{Case studies: Particle Datasets. The dashed reference line denotes the target confidences, and only methods with coverage curves at or above this line achieve the target coverages. The shaded region of each curve is the $\pm$ 1 standard error region. In Area graphs, lower curves indicate better performance.}
    \label{fig: particle datasets}
\end{figure*}
\begin{figure*}[htbp]
    \centering
    \includegraphics[width=1\linewidth]{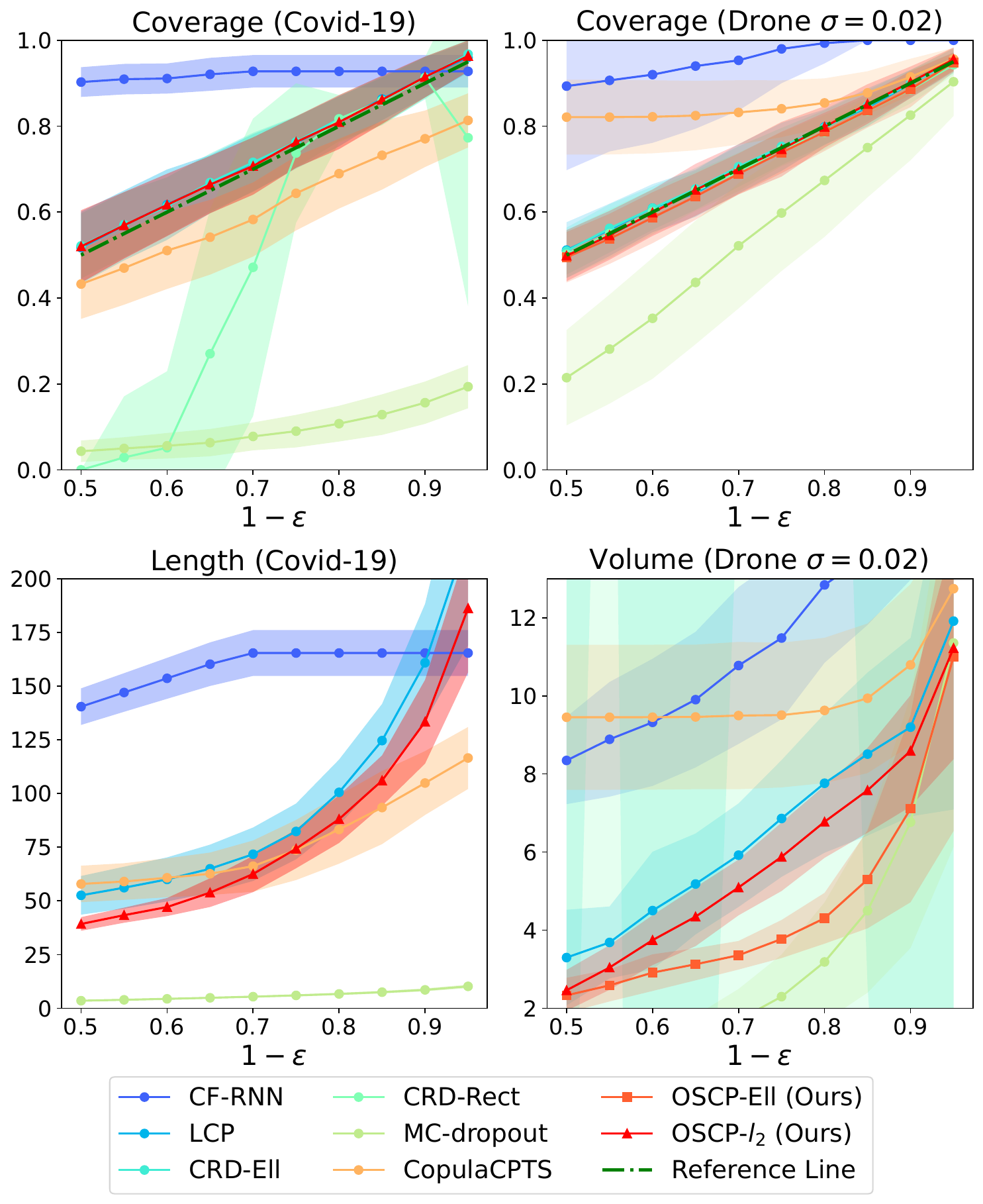}
    \caption{Case studies: Drone Dataset \& Covid-19 Dataset. The dashed reference line denotes the target confidences, and only methods with coverage curves at or above this line achieve the target coverages. The shaded region of each curve is the $\pm$ 1 standard error region. In Length/Volume graphs, lower curves indicate better performance.}
    \label{fig: drone n covid datasets}
\end{figure*}

\begin{table*}[htbp] 
  \centering 
\caption{Detailed performance comparison of all baselines with confidence levels from 85\% to 95\%}
\label{Table: Performance Comparison}
  
  \begin{subtable}{\textwidth} 
    \centering
    \caption{Particle trajectory simulation ($\sigma=0.01$): $d=2$, $T=25$}
    \resizebox{\textwidth}{!}{
    \begin{tabular}{ccccccc}
    \toprule
      \multicolumn{1}{c}{} & \multicolumn{2}{
      c}{Target Coverage: 85\%} & \multicolumn{2}{c}{Target Coverage: 90\%} & \multicolumn{2}{c}{Target Coverage: 95\%} \\ 
      \cmidrule(l){2-7}
   
      Method & Coverage & Area & Coverage & Area & Coverage & Area \\
      \midrule

OSCP-$\ell_2$ (Ours) & \textbf{85.8} {\scriptsize$\pm$ 2.7} & \textbf{0.59} {\scriptsize$\pm$ 0.03 }& \textbf{90.5} {\scriptsize$\pm$ 2.1} & \textbf{0.69} {\scriptsize$\pm$ 0.04} & \textbf{95.4} {\scriptsize$\pm$ 1.6 }& \textbf{0.88} {\scriptsize$\pm$ 0.07}\\

OSCP-$Ell$ (Ours) & \textbf{85.3} {\scriptsize $\pm$ 2.8} & \textbf{0.59} {\scriptsize$\pm$ 0.04} & \textbf{90.2} {\scriptsize$\pm$ 2.4} & \textbf{0.68} {\scriptsize$\pm$ 0.05} & \textbf{95.3} {\scriptsize$\pm$ 1.7} & \textbf{0.87} {\scriptsize$\pm$ 0.07}\\

LCP \cite{cleaveland2024conformal} & 85.4 {\scriptsize$\pm$ 2.2 }& 0.69 {\scriptsize$\pm$ 0.07 }& 90.1 {\scriptsize$\pm$ 2.3 }& 0.81 {\scriptsize$\pm$ 0.08 }& 95.5 {\scriptsize$\pm$ 1.4} & 1.05 {\scriptsize$\pm$ 0.12}\\

CopulaCPTS \cite{sun2024copula} & 85.3 {\scriptsize$\pm$ 3.3 }& 0.74 {\scriptsize$\pm$ 0.11} & 89.8 {\scriptsize$\pm$ 3.0} & 0.91 {\scriptsize$\pm$ 0.3 }&  \textcolor{gray}{94.4 {\scriptsize$\pm$ 2.1 }}&  \textcolor{gray}{1.7 {\scriptsize$\pm$ 1.16}}\\

CRD-Ell \cite{Tumu2024multi-modal} & 85.3 {\scriptsize$\pm$ 3.0} & 0.93 {\scriptsize$\pm$ 0.23} & 90.2 {\scriptsize$\pm$ 2.4 }& 1.02 {\scriptsize$\pm$ 0.22}& 95.6 {\scriptsize$\pm$ 1.4} & 1.24 {\scriptsize$\pm$ 0.19}\\

CRD-Rect \cite{Tumu2024multi-modal}& 92.7 {\scriptsize$\pm$ 4.2} & 1.52 {\scriptsize$\pm$ 0.41 }& 94.5 {\scriptsize$\pm$ 3.3} & 1.55 {\scriptsize$\pm$ 0.41} & 96.1 {\scriptsize$\pm$ 2.2 }& 1.59 {\scriptsize$\pm$ 0.39}\\

CF-RNN \cite{stankeviciute2021conformal} & 97.0 {\scriptsize$\pm$ 1.7 }& 1.19 {\scriptsize$\pm$ 0.22 }& 98.3 {\scriptsize$\pm$ 1.5}& 2.04 {\scriptsize$\pm$ 1.03}& 99.4 {\scriptsize$\pm$ 0.7} & 4.12 {\scriptsize$\pm$ 1.46}\\

MC-dropout \cite{gal2016dropout}& 87.7 {\scriptsize$\pm$ 6.8 }& 1.8 {\scriptsize$\pm$ 0.06 }& 91.4 {\scriptsize$\pm$ 5.2 }& 2.34 {\scriptsize$\pm$ 0.08} & 94.9{\scriptsize $\pm$ 3.6 }& 3.33 {\scriptsize$\pm$ 0.11}\\
       \bottomrule
    \end{tabular}}
  
  \end{subtable}

  \vspace{1em} 

  \begin{subtable}{\linewidth}
    \centering
    \caption{Particle trajectory simulation ($\sigma=0.05$): $d=2$, $T=25$}

    \resizebox{\linewidth}{!}{
    \begin{tabular}{ccccccc}
    \toprule
      \multicolumn{1}{c}{} & \multicolumn{2}{
      c}{Target Coverage: 85\%} & \multicolumn{2}{c}{Target Coverage: 90\%} & \multicolumn{2}{c}{Target Coverage: 95\%} \\ 
      \cmidrule(l){2-7}

      Method & Coverage & Area & Coverage & Area & Coverage & Area \\
      \midrule

      OSCP-$\ell_2$ (Ours) & \textbf{85.4} {\scriptsize $\pm$ 2.4} & \textbf{3.85} {\scriptsize $\pm$ 0.18} & \textbf{90.6} {\scriptsize $\pm$ 2.4} & \textbf{4.43} {\scriptsize $\pm$ 0.27} & \textbf{95.6} {\scriptsize $\pm$ 1.7} & \textbf{5.6} {\scriptsize $\pm$ 0.4}\\

       OSCP-$Ell$ (Ours) & \textbf{86.1} {\scriptsize $\pm$ 2.5} & \textbf{3.81} {\scriptsize $\pm$ 0.16} & \textbf{90.5} {\scriptsize $\pm$ 2.2} & \textbf{4.27} {\scriptsize $\pm$ 0.17} & \textbf{95.2} {\scriptsize $\pm$ 1.7} & \textbf{5.16} {\scriptsize $\pm$ 0.31}\\
      
      LCP \cite{cleaveland2024conformal} & 85.4 {\scriptsize$\pm$ 2.4} & 4.43 {\scriptsize$\pm$ 0.29} & 90.5 {\scriptsize$\pm$ 1.8} & 5.07 {\scriptsize$\pm$ 0.3} & 95.5 {\scriptsize$\pm$ 1.4} & 6.41 {\scriptsize$\pm$ 0.67}\\
      
      CopulaCPTS \cite{sun2024copula}& 85.6 {\scriptsize$\pm$ 2.8} & 4.44 {\scriptsize$\pm$ 0.46} & 90.0 {\scriptsize$\pm$ 2.6} & 5.43 {\scriptsize$\pm$ 0.86} &  \textcolor{gray}{93.4 {\scriptsize$\pm$ 2.2}} &  \textcolor{gray}{6.46 {\scriptsize$\pm$ 1.28}}\\
      
      CRD-Ell \cite{Tumu2024multi-modal} & 85.6 {\scriptsize$\pm$ 2.8} & 4.88 {\scriptsize$\pm$ 0.44} & 90.5 {\scriptsize$\pm$ 2.0} & 5.5 {\scriptsize$\pm$ 0.58 }& 95.6 {\scriptsize$\pm$ 1.1} & 7.11 {\scriptsize$\pm$ 0.71}\\
      
      CRD-Rect \cite{Tumu2024multi-modal} & 92.0 {\scriptsize$\pm$ 3.0} & 7.09 {\scriptsize$\pm$ 0.82}  & 93.2 {\scriptsize$\pm$ 2.4} & 7.3 {\scriptsize$\pm$ 1.05} & 95.4 {\scriptsize$\pm$ 1.8} & 7.9 {\scriptsize$\pm$ 1.06}\\
      
        CF-RNN \cite{stankeviciute2021conformal}& 92.6 {\scriptsize$\pm$ 1.5} & 5.2 {\scriptsize$\pm$ 0.29} & 95.0 {\scriptsize$\pm$ 1.5} & 6.08 {\scriptsize$\pm$ 0.67} & 97.6 {\scriptsize$\pm$ 1.1} & 8.41 {\scriptsize$\pm$ 1.29}\\
        
        MC-dropout \cite{gal2016dropout}& \textcolor{gray}{33.8 {\scriptsize$\pm$ 4.8}} & \textcolor{gray}{1.81 {\scriptsize$\pm$ 0.07}  }& \textcolor{gray}{42.7 {\scriptsize$\pm$ 5.1} }&\textcolor{gray}{ 2.36 {\scriptsize$\pm$ 0.09}} &\textcolor{gray}{ 55.5 {\scriptsize$\pm$ 5.2 }}& \textcolor{gray}{3.35 {\scriptsize$\pm$ 0.13}}\\

      \bottomrule
    \end{tabular}}
  \end{subtable}

  \vspace{1em}

  \begin{subtable}{\textwidth}
    \centering
    \caption{Drone trajectory simulation ($\sigma=0.02$): $d=3$, $T=10$}
\resizebox{\textwidth}{!}{
    \begin{tabular}{ccccccc}
    \toprule
      \multicolumn{1}{c}{} & \multicolumn{2}{
      c}{Target Coverage: 85\%} & \multicolumn{2}{c}{Target Coverage: 90\%} & \multicolumn{2}{c}{Target Coverage: 95\%} \\ 
      \cmidrule(l){2-7}

      Method & Coverage & Volume & Coverage & Volume & Coverage & Volume \\
      \midrule
      
OSCP-$\ell_2$ (Ours) & \textbf{85.2} {\scriptsize$\pm$ 4.6 }& \textbf{7.58 }{\scriptsize$\pm$ 1.09 }& \textbf{90.2} {\scriptsize$\pm$ 3.7 }& \textbf{8.59} {\scriptsize$\pm$ 1.43} & \textbf{95.5} {\scriptsize$\pm$ 2.8 }& \textbf{11.22} {\scriptsize$\pm$ 2.84}\\

OSCP-$Ell$ (Ours) & \textcolor{gray}{83.6 {\scriptsize$\pm$ 3.9}} &  \textcolor{gray}{5.3 {\scriptsize$\pm$ 1.25}} &  \textcolor{gray}{88.6 {\scriptsize$\pm$ 4.2}} &  \textcolor{gray}{7.11 {\scriptsize$\pm$ 2.4}} &  \textcolor{gray}{94.8 {\scriptsize$\pm$ 2.3}} &  \textcolor{gray}{11.01 {\scriptsize$\pm$ 4.47}}\\

LCP \cite{cleaveland2024conformal} & 84.4 {\scriptsize$\pm$ 4.5 }& 8.51 {\scriptsize$\pm$ 2.09} & 89.9 {\scriptsize$\pm$ 3.2 }& 9.2 {\scriptsize$\pm$ 2.29} & 95.2 {\scriptsize$\pm$ 2.9} & 11.91 {\scriptsize$\pm$ 4.83}\\

CopulaCPTS \cite{sun2024copula} & 87.8 {\scriptsize$\pm$ 5.0} & 9.94 {\scriptsize$\pm$ 1.92 }& 91.5 {\scriptsize$\pm$ 4.2 }& 10.79 {\scriptsize$\pm$ 2.07} & 95.6 {\scriptsize$\pm$ 2.6 }& 12.74 {\scriptsize$\pm$ 1.95}\\

CRD-Ell \cite{Tumu2024multi-modal} & 84.7 {\scriptsize$\pm$ 3.7} & 163.64   {\scriptsize$\pm$ 161.07} & 90.0 {\scriptsize$\pm$ 3.1} & 207.23 {\scriptsize$\pm$ 221.37} & 95.1 {\scriptsize$\pm$ 2.1} & 305.0 {\scriptsize$\pm$ 537.19}\\

CRD-Rect \cite{Tumu2024multi-modal} & \textcolor{gray}{83.6 {\scriptsize$\pm$ 3.5} }& \textcolor{gray}{535.7 {\scriptsize$\pm$ 539.58 }}& \textcolor{gray}{88.5 {\scriptsize$\pm$ 3.2}} & \textcolor{gray}{925.28 {\scriptsize$\pm$ 1341.29 }}&  \textcolor{gray}{94.3{\scriptsize $\pm$ 2.4}} &  \textcolor{gray}{1634.36 {\scriptsize$\pm$ 3431.92}}\\

CF-RNN \cite{stankeviciute2021conformal} & 100.0 {\scriptsize$\pm$ 0.0 }& 13.61{\scriptsize $\pm$ 1.79} & 100.0 {\scriptsize$\pm$ 0.0 }& 14.27 {\scriptsize$\pm$ 1.29} & 100.0 {\scriptsize$\pm$ 0.0} & 14.95 {\scriptsize$\pm$ 0.21}\\

MC-dropout \cite{gal2016dropout} & \textcolor{gray}{75.0 {\scriptsize$\pm$ 11.6}} & \textcolor{gray}{4.5 {\scriptsize$\pm$ 2.1}} & \textcolor{gray}{82.6 {\scriptsize$\pm$ 10.5}} & \textcolor{gray}{6.76 {\scriptsize$\pm$ 3.22}} & \textcolor{gray}{90.3 {\scriptsize$\pm$ 8.0}} & \textcolor{gray}{11.35 {\scriptsize$\pm$ 5.22}}\\

\bottomrule

    \end{tabular}}
  \end{subtable}

   \vspace{1em} 

  \begin{subtable}{\textwidth}
    \centering
    \caption{Covid-19 daily cases: $d=1$, $T=50$}
    \resizebox{\textwidth}{!}{
    \begin{tabular}{ccccccc}
    \toprule
      \multicolumn{1}{c}{} & \multicolumn{2}{
      c}{Target Coverage: 85\%} & \multicolumn{2}{c}{Target Coverage: 90\%} & \multicolumn{2}{c}{Target Coverage: 95\%} \\ 
      \cmidrule(l){2-7}

      Method & Coverage & Length & Coverage & Length & Coverage & Length \\
      \midrule

OSCP-$\ell_2$ (Ours) & \textbf{86.2} {\scriptsize$\pm$ 5.6} & \textbf{106.01} {\scriptsize$\pm$ 11.72} & \textbf{91.4} {\scriptsize$\pm$ 4.8} & \textbf{133.4} {\scriptsize$\pm$ 19.5} & \textbf{96.4} {\scriptsize$\pm$ 3.6} & \textbf{186.23} {\scriptsize$\pm$ 29.83}\\

LCP \cite{cleaveland2024conformal}& 86.4 {\scriptsize$\pm$ 4.9} & 124.62{\scriptsize $\pm$ 17.06} & 91.4 {\scriptsize$\pm$ 4.7 }& 160.87 {\scriptsize$\pm$ 27.38} & 96.0 {\scriptsize$\pm$ 3.6} & 220.38{\scriptsize $\pm$ 51.33}\\

CopulaCPTS \cite{sun2024copula} & \textcolor{gray}{73.2 {\scriptsize$\pm$ 8.0}} & \textcolor{gray}{93.45 {\scriptsize$\pm$ 17.03}} & \textcolor{gray}{77.1 {\scriptsize$\pm$ 6.6}} & \textcolor{gray}{104.86 {\scriptsize$\pm$ 14.86}} & \textcolor{gray}{81.4 {\scriptsize$\pm$ 6.3}} & \textcolor{gray}{116.53 {\scriptsize$\pm$ 14.42}}\\

CRD-Ell \cite{Tumu2024multi-modal}  & 85.9 {\scriptsize$\pm$ 5.6} & 1206.44 {\scriptsize$\pm$ 238.9} & 91.2 {\scriptsize$\pm$ 4.9} & 1486.12 {\scriptsize$\pm$ 365.38} & 96.7 {\scriptsize$\pm$ 3.3} & 2074.58 {\scriptsize$\pm$ 549.95}\\

CRD-Rect \cite{Tumu2024multi-modal}  & 85.9 {\scriptsize$\pm$ 5.7} & 798.16 {\scriptsize$\pm$ 167.12} & 91.3 {\scriptsize$\pm$ 4.8} & 999.14 {\scriptsize$\pm$ 277.5} & \textcolor{gray}{77.3 {\scriptsize$\pm$ 39.2 }}& \textcolor{gray}{N/A {\scriptsize$\pm$ N/A}}\\

CF-RNN \cite{stankeviciute2021conformal} & 92.8 {\scriptsize$\pm$ 3.8} & 165.4 {\scriptsize$\pm$ 10.66} & 92.8 {\scriptsize$\pm$ 3.8} & 165.4 {\scriptsize$\pm$ 10.66} &\textcolor{gray}{ 92.8 {\scriptsize$\pm$ 3.8 }}& \textcolor{gray}{165.4 {\scriptsize$\pm$ 10.66}}\\

MC-dropout \cite{gal2016dropout} & \textcolor{gray}{12.8 {\scriptsize$\pm$ 4.7}} & \textcolor{gray}{7.48 {\scriptsize$\pm$ 0.61}} & \textcolor{gray}{15.6 {\scriptsize$\pm$ 4.9} }& \textcolor{gray}{8.55 {\scriptsize$\pm$ 0.7}} & \textcolor{gray}{19.3 {\scriptsize$\pm$ 5.0}} & \textcolor{gray}{10.19 {\scriptsize$\pm$ 0.83}}\\

      \bottomrule

        \end{tabular}}
  \end{subtable}

\end{table*}

\end{document}